\documentclass{amsart}

\usepackage{amsmath}
\usepackage{amsthm}
\usepackage{amsfonts}
\usepackage{dsfont}

\newcommand{\reals}{\mathbb{R}}

\newcommand{\complex}{\mathbb{C}}

\newcommand{\bracketb}[1]{\Big[#1\Big]}

\newcommand{\angles}[1]{\left\langle #1 \right\rangle}
\newcommand{\pb}[1]{\left\{#1\right\}}

\newcommand{\paraa}[1]{\big(#1\big)}
\newcommand{\parab}[1]{\Big(#1\Big)}

\newcommand{\diag}{\operatorname{diag}}

\newtheorem{theorem}{Theorem}[section]

\newtheorem{lemma}[theorem]{Lemma}
\newtheorem{proposition}[theorem]{Proposition}

\theoremstyle{definition}
\newtheorem{definition}[theorem]{Definition}
\theoremstyle{remark}

\numberwithin{equation}{section}

\newcommand{\Tr}{\operatorname{Tr}}

\newcommand{\A}{\mathcal{A}}
\newcommand{\B}{\mathcal{B}}
\renewcommand{\P}{\mathcal{P}}

\newcommand{\R}{\mathcal{R}}
\newcommand{\J}{\mathcal{J}}

\renewcommand{\d}{\partial}
\newcommand{\TSigma}{T\Sigma}
\newcommand{\eps}{\varepsilon}

\newcommand{\nablab}{\bar{\nabla}}

\newcommand{\Gammab}{\bar{\Gamma}}
\newcommand{\Rb}{\bar{R}}

\newcommand{\D}{\mathcal{D}}
\newcommand{\Dt}{\widetilde{\D}}

\newcommand{\KP}{K\"ahler--Poisson}
\newcommand{\nablah}{\hat{\nabla}}
\newcommand{\nablat}{\widetilde{\nabla}}
\newcommand{\nh}{\nablah}
\newcommand{\Der}{\operatorname{Der}}
\newcommand{\X}{\mathcal{X}}
\newcommand{\N}{\mathcal{N}}
\renewcommand{\J}{\mathcal{J}}

\renewcommand{\div}{\operatorname{div}}
\newcommand{\ints}{\int_\A}

\title[]{On the geometry of K\"ahler--Poisson structures}

\author{Joakim Arnlind}
\address[Joakim Arnlind]{Dept. of Math.\\
Link\"oping University\\
581 83 Link\"oping\\
Sweden}
\email{joakim.arnlind@liu.se}

\author{Gerhard Huisken}
\address[Gerhard Huisken]{Max Planck Institute for Gravitational Physics\\
Am M\"uhlenberg 1\\
D-14476 Golm\\
Germany}
\email{gerhard.huisken@aei.mpg.de}

\thanks{}

\subjclass[2000]{}

\begin{document}

\begin{abstract}
  We prove that the Riemannian geometry of almost K\"ahler manifolds
  can be expressed in terms of the Poisson algebra of smooth functions
  on the manifold. Subsequently, K\"ahler--Poisson algebras are
  introduced, and it is shown that a corresponding purely algebraic
  theory of geometry and curvature can be developed. As an
  illustration of the new concepts we give an algebraic proof of the
  statement that a bound on the (algebraic) Ricci curvature induces a
  bound on the eigenvalues of the (algebraic) Laplace operator, in
  analogy with the well-known theorem in Riemannian geometry.

  As the correspondence between Poisson brackets of smooth functions
  and commutators of operators lies at the heart of quantization, a
  purely Poisson algebraic proof of, for instance, such a ``Gap
  Theorem'', might lead to an understanding of spectral properties in
  a corresponding quantum mechanical system.
\end{abstract}

\maketitle

\section{Introduction}

\noindent In a series of papers, the possibility of expressing
differential geometry of Riemannian submanifolds as Nambu-algebraic
expressions in the function algebra has been investigated
\cite{ahh:discretegb,ahh:nsurface,ahh:psurface,ahh:nambudiscrete}.
More precisely, it was shown that on a $n$-dimensional submanifold
$\Sigma$, geometric objects can be written in terms of a $n$-ary
alternating multi-linear map acting on the embedding functions.  One
of the original motivations for studying the problem came from matrix
regularizations of surfaces in the context of ``Membrane Theory''
(cp. \cite{h:phdthesis}), where smooth functions are mapped to
hermitian matrices such that the Poisson bracket of functions
correspond to the commutator of matrices (as the matrix dimension
becomes large).  In this context, matrices corresponding to the
embedding coordinates of a surface arise as solutions to equations,
which contain matrices associated to surfaces of arbitrary genus. In
order to identify the topology of a solution, it is desirable to be
able to compute geometric invariants in terms of the embedding
matrices and their commutators. This was illustrated in
\cite{ahh:discretegb} where formulas for the discrete scalar curvature
and the discrete genus were presented.

A natural generalization to higher dimensional manifolds is to require
that the $n$-ary multi linear bracket corresponds to a $n$-ary
``commutator'' of matrices. However, there is no natural candidate for
such a $n$-ary map, and it is hard to construct explicit realizations.
One may then ask the following question: Is there a particular class
of manifolds (of dimension greater than two) for which one can use a
Poisson bracket on the space of smooth functions to express geometric
quantities? In the following we shall demonstrate that almost K\"ahler
manifolds provide a context where an affirmative answer can be given.

Once a theory of Riemannian differential geometry in terms of Poisson
brackets has been developed, one wonders whether the obtained formulas
make sense in an arbitrary Poisson algebra? That is, can one use the
results to develop a theory of Poisson algebraic geometry? As
expected, this is not possible in general, and one might search for an
intrinsic definition of algebras for which a theory of differential
geometry can be implemented. It turns out that starting from the
simple assumption that the square of the Poisson bivector is
(proportional to) a projection operator, one can in fact reproduce
several standard results with purely algebraic methods. For instance,
if the sectional curvature is independent of the choice of tangent
plane, then it follows that the sectional curvature is in the center
of the Poisson algebra (i.e. a ``Poisson-constant''). Moreover, one
can prove that a bound on the Ricci curvature induces a bound on the
eigenvalues of the Laplace operator. This framework opens up for an
algebraic treatment of Riemannian geometry, which have many potential
applications. In particular, as the correspondence between Poisson
brackets and operator commutators lies at the heart of quantization,
the presented results should have an impact on the quantization of
geometrical systems. Moreover, concerning the original motivation, our
result fits nicely with the fact that for any (quantizable) compact
K\"ahler manifold there exists a matrix regularization
\cite{bms:toeplitz}, and it suggests a way to define matrix
regularizations without any reference to a manifold. Finally, we hope
that it is possible to extend the results to non-commutative Poisson
algebras, providing an interesting approach to non-commutative
geometry.

The structure of the paper is as follows. In Section
\ref{sec:preliminaries} we recall some concepts for Riemannian
submanifolds, and in Section \ref{sec:KPstructures} \KP\ structures
are introduced together with some basic results. Section
\ref{sec:diffgeoKP} is devoted to the reformulation of Riemannian
geometry of K\"ahler submanifolds in terms of Poisson brackets of the
embedding coordinates. In Section \ref{sec:KPalgebras} we formulate a
theory of algebraic Riemannian geometry for \KP\ algebras and show
that analogues of differential geometric theorems can be proven in the
purely algebraic setting.

\section{Preliminaries}\label{sec:preliminaries}

\noindent Let $(M,\eta)$ be a Riemannian manifold of dimension $m$,
and let $(\Sigma,g)$ be a $n$-dimensional submanifold of $M$ with
induced metric $g$. Given local coordinates $x^1,\ldots,x^m$ on $M$,
we consider $\Sigma$ as embedded in $M$ via $x^i(u^1,\ldots,u^n)$
where $u^1,\ldots,u^n$ are local coordinates on $\Sigma$. Indices
$i,j,k,\ldots$ run from $1$ to $m$ and indices $a,b,c,\ldots$ run from
$1$ to $n$. The covariant derivative on $M$ is denoted by $\nablab$
(with Christoffel symbols $\Gammab^i_{jk}$) and the covariant
derivative on $\Sigma$ by $\nabla$ (with Christoffel symbols
$\Gamma^a_{bc}$).  The tangent space $\TSigma$ is regarded as a
subspace of the tangent space $TM$ and at each point of $\Sigma$ one
can choose $e_a=(\d_ax^i)\d_i$ as basis vectors in $\TSigma$, and in
this basis we define $g_{ab}=\eta(e_a,e_b)$. Moreover, we introduce an
orthonormal basis of $\TSigma^\perp$, given by the vectors $N_A$ for
$A=1,\ldots,p$.

The formulas of Gauss and Weingarten split the covariant derivative in
$M$ into tangential and normal components as
\begin{align}
  &\nablab_X Y = \nabla_X Y + \alpha(X,Y)\label{eq:GaussFormula}\\
  &\nablab_XN_A = -W_A(X) + D_XN_A\label{eq:WeingartenFormula}
\end{align}
where $X,Y\in \TSigma$ and $\nabla_X Y$, $W_A(X)\in\TSigma$ and
$\alpha(X,Y)$, $D_XN_A\in\TSigma^\perp$. By expanding $\alpha(X,Y)$ in
the basis $\{N_1,\ldots,N_p\}$ one can write (\ref{eq:GaussFormula}) as
\begin{align}
  &\nablab_X Y = \nabla_X Y + \sum_{A=1}^ph_A(X,Y)N_A,\label{eq:GaussFormulah}
\end{align}
and we set $h_{A,ab} = h_A(e_a,e_b)$. From the above equations one derives the relation
\begin{align}
  h_{A,ab} &= -\eta\paraa{e_a,\nablab_b N_A},
\end{align}
as well as Weingarten's equation
\begin{align}
  h_A(X,Y) = \eta\paraa{W_A(X),Y},  
\end{align}
which implies that $(W_A)^a_b = g^{ac}h_{A,cb}$.

From formulas (\ref{eq:GaussFormula}) and (\ref{eq:WeingartenFormula})
one also obtains Gauss' equation, i.e. an expression for the curvature $R$
of $\Sigma$ in terms of the curvature $\Rb$ of $M$, as
\begin{equation}\label{eq:GaussEquation}
  \begin{split}
    R(X,Y,Z,V) &= \Rb(X,Y,Z,V)+\eta\paraa{\alpha(X,Z),\alpha(Y,V)}\\
    &\qquad-\eta\paraa{\alpha(X,V),\alpha(Y,Z)}
  \end{split}
\end{equation}
where $R(X,Y,Z,V)\equiv\eta\paraa{R(Z,V)Y,X}$ and $X,Y,Z,V\in\TSigma$.

\section{K\"ahler  Poisson structures}\label{sec:KPstructures}

\noindent In case $\Sigma$ is a surface, a generic
Poisson structure may be written as
\begin{align*}
  \{f_1,f_2\} = \frac{1}{\rho}\eps^{ab}\paraa{\d_af_1}\paraa{\d_bf_2}
\end{align*}
for some density $\rho$, where $\eps^{ab}$ is the totally
anti-symmetric Levi-Civita symbol. Setting $\theta^{ab} =
\eps^{ab}/\rho$ and $g=\det(g_{ab})$, one notes that
\begin{align}\label{eq:surfacePoissonRel}
  \frac{g}{\rho^2}g^{ab} = \theta^{ap}\theta^{bq}g_{pq}
  = \frac{1}{\rho}\eps^{ap}\frac{1}{\rho}\eps^{bq}g_{pq}
\end{align}
since the right hand side is simply the cofactor expansion of the
inverse metric. If $\Sigma$ is embedded in $M$ via the embedding
coordinates $x^i$, the above relation allows one to write many of the
differential geometric objects of $\Sigma$ in terms of $\{x^i,x^j\}$
\cite{ahh:psurface}. It was also shown that for higher dimensional
submanifolds one obtains a similar description using $n$-ary Nambu
brackets \cite{ahh:nambudiscrete,ahh:nsurface}. One may ask if there
is a class of $n$-dimensional submanifolds (with $n>2$) that allows for a
description in terms of Poisson brackets? It turns out that relation
(\ref{eq:surfacePoissonRel}) plays a key role in the answer to this
question.  Therefore, we make the following definition.

\begin{definition}
  Let $(\Sigma,g)$ be a Riemannian manifold and let $\theta$ be a
  Poisson bivector on $\Sigma$. If there exists a strictly positive
  $\gamma\in C^{\infty}(\Sigma)$ such that
  \begin{equation}\label{eq:projectivePoissonStructure}
    \gamma^2g^{-1}(\sigma,\tau) = g\paraa{\theta(\sigma),\theta(\tau)}
  \end{equation}
  for all 1-forms $\sigma,\tau$, then $\theta$ is called \emph{an almost
  \KP\ structure on} $(\Sigma,g)$. In local coordinates, the
  above relation is equivalent to
  \begin{align}\label{eq:KPcondlocal}
    \gamma^2g^{ab} = \theta^{ap}\theta^{bq}g_{pq}.
  \end{align}
  Moreover, we call the function $\gamma^2$ the \emph{characteristic
    function} of the almost \KP\ structure.
\end{definition}

\noindent Note that such ``self-dual'' Poisson structures, and
corresponding geometric formulas, have also been studied in
the context of matrix models for gravity
\cite{bs:curvatureMatrix,bs:curvatureMatrixII}.

The reason for calling it an almost \KP\ structure is that the Poisson
structure induced from the K\"ahler form always fulfills
(\ref{eq:projectivePoissonStructure}), and that the existence of an
almost \KP\ structure implies that $\Sigma$ is an almost K\"ahler
manifold (with respect to a rescaled metric, as we shall see). Let us
first prove that an almost K\"ahler manifold has an almost \KP\
structure.

\begin{proposition}\label{prop:aKMgivesKP}
  Let $(\Sigma,g)$ be an almost K\"ahler manifold with K\"ahler form
  $\omega$. The Poisson bracket on $C^\infty(\Sigma)$, given as
  \begin{align}\label{eq:KahlerFormPB}
    \{f_1,f_2\} = \omega(X_{f_1},X_{f_2})
  \end{align}
  where $\omega(X_f,Y)=df(Y)$ for all $Y\in T\Sigma$, defines an
  almost \KP\ structure on $\Sigma$ with characteristic function
  $1$.
\end{proposition}

\begin{proof}
  It is a standard fact that formula (\ref{eq:KahlerFormPB}) defines a
  Poisson bracket for any symplectic form $\omega$, and the Poisson
  bivector is the inverse of $\omega$. Let us now show that formula
  (\ref{eq:KPcondlocal}) is fulfilled with $\gamma^2=1$.

  Let $J^a_b$ be the almost complex structure of $\Sigma$. The
  K\"ahler form is then written as $\omega_{ab}=g_{ac}J^c_b$. The
  inverse of $\omega$ is then computed to be
  \begin{align*}
    \theta^{ab}=-\omega^{ab}=\omega_{pq}g^{pa}g^{qb} = -J^a_cg^{cb}.
  \end{align*}
  Using the fact that $g$ is an hermitian metric one obtains
  \begin{align*}
    \theta^{ap}\theta^{bq}g_{pq} = J^a_rg^{rp}J^b_sg^{sq}g_{pq}
    = g^{rp}J^a_rJ^b_p = g^{ab},
  \end{align*}
  which proves that $\theta$ is an almost \KP\ structure with $\gamma^2=1$.
\end{proof}

\noindent On the other hand, it immediately follows from
(\ref{eq:KPcondlocal}) that ${\J^a}_b\equiv\gamma^{-1}{\theta^a}_b$ is
an almost complex structure, i.e. $\J^2(X)=-X$ for all
$X\in\TSigma$. However, it is not necessarily integrable. Given a
Riemannian manifold $(\Sigma,g)$ with an almost \KP\ structure
$\theta$ one can define $\tilde{g} = \gamma^{-1}g$, and since $\theta$
is \KP\ structure it is easy to see that $\tilde{g}$ is an hermitian
metric with respect to the complex structure defined
above. Furthermore, one defines the K\"ahler form $\Omega$ as
\begin{align*}
  \Omega(X,Y) = \tilde{g}(\J(X),Y), 
\end{align*}
which implies that $\Omega_{ab}=-\frac{1}{\gamma^2}\theta_{ab}$. Since
$\theta$ is an almost \KP\ structure, it follows that $\Omega$ is the
inverse of $\theta$, i.e. $\Omega_{ab}\theta^{bc} =
\delta_a^c$, which implies that $d\Omega=0$ since $\theta$
satisfies the Jacobi identity. Hence, $(\Sigma,\tilde{g},\J)$ is an
almost K\"ahler manifold. Thus, we have proved the following statement: 

\begin{proposition}
  Let $(\Sigma,g)$ be a Riemannian manifold such that there exists a
  \KP\ structure $\theta$ with characteristic function
  $\gamma^2$. Then $(\Sigma,\gamma^{-1}g,\J)$ is an almost K\"ahler
  manifold, with ${\J^a}_b=\gamma^{-1}{\theta^{a}}_b$.
\end{proposition}

\section{Differential geometry of embedded \KP\ structures}\label{sec:diffgeoKP}

\noindent In the following we shall assume that $\Sigma$ is a submanifold of
$M$ and that there exists an almost \KP\ structure $\theta$ on
$\Sigma$ with characteristic function $\gamma^2$. The submanifold
$\Sigma$ is embedded in $M$ via the embedding coordinates
$x^1,\ldots,x^m$.  Our main goal is to express geometric properties of
$\Sigma$ in terms of Poisson brackets of embedding coordinates (and
components of normal vectors) and the geometric objects of the ambient
manifold $M$, such as the metric $\eta_{ij}$ and the Christoffel
symbols $\Gammab^{i}_{jk}$. In particular, derivatives should only
appear as Poisson brackets.

Let us define
\begin{align}
  \P^{ij} = \pb{x^i,x^j},
\end{align}
which will also be considered as a map $\P:TM\to TM$ through
$\P(X)=\P^{ij}\eta_{jk}X^k\d_i$. Furthermore, for all $u\in
C^\infty(\Sigma)$ we introduce
\begin{align}
  \D^i(u) = \frac{1}{\gamma^2}\{u,x^k\}\{x^i,x^l\}\eta_{kl}
\end{align}
as well as
\begin{align}
  \D^{ik} = \D^i(x^k)\qquad\nablah^i = \D^{ik}\nablab_k
  \qquad\nablah_XY=X_i\nablah^iY,  
\end{align}
where $\nablab$ is the covariant derivative on $M$. As opposed to
$\nablab$, the derivative $\nablah$ can be expressed in
terms of Poisson brackets; i.e.
\begin{align}
  \nablah^iX^k = \D^i(X^k)+\D^i(x^l)\Gammab^k_{lm}X^m.
\end{align}
For upcoming calculations, it is convenient to note that
\begin{equation}\label{eq:Diformula}
  \begin{split}
    \D^i(u) &= \frac{1}{\gamma^2}\theta^{ab}\paraa{\d_ax^i}\paraa{\d_bx^k}\eta_{kl}
    \theta^{pq}\paraa{\d_pu}\paraa{\d_qx^l}\
    =\frac{1}{\gamma^2}\theta^{ab}\theta^{pq}g_{bq}\paraa{\d_ax^i}\paraa{\d_pu}\\
    &=g^{ap}\paraa{\d_ax^i}\paraa{\d_p u},
  \end{split}
\end{equation}
which is independent of $\theta$.  

Let us gather some properties of $\P$ in the next proposition.

\begin{proposition}\label{prop:Pproperties}
  Let $\J=\gamma^{-1}\theta$ denote the associated almost complex structure on
  $\Sigma$. For all $X\in TM$ and $Y\in\TSigma$ it holds that
  \begin{align}
    &\P(Y) = \gamma\J(Y)\label{eq:PYeqJ}\\
    &\P^2(X) = -\gamma^2\eta(X,e_a)g^{ab}e_b\label{eq:PsqX}\\
    &\Tr\P^2\equiv(\P^2)^i_i = -n\gamma^2\label{eq:TrPsq}.
  \end{align}
  In particular, it follows that $-\gamma^{-2}\P^2$ is the orthogonal
  projection onto $\TSigma$.
\end{proposition}

\begin{proof}
  Let us start by proving (\ref{eq:PYeqJ}). By definition we obtain
  \begin{align*}
    \P(Y) &= \theta^{ab}\paraa{\d_ax^i}\paraa{\d_bx^j}Y_j\d_i
    =\theta^{ab}\paraa{\d_ax^i}\paraa{\d_bx^j}Y^c\paraa{\d_cx^k}\eta_{jk}\d_i\\
    &=\theta^{ab}\paraa{\d_a x^i}g_{bc}Y^c\d_i
    =\theta^a_c\paraa{\d_ax^i}Y^c\d_i
    =\theta^a_cY^c\d_a = \gamma\J(Y).
  \end{align*}
  Let us also prove formula (\ref{eq:PsqX}). One obtains
  \begin{align*}
    \P^2(X) &=
    \theta^{ab}\paraa{\d_ax^i}\paraa{\d_bx^j}\eta_{jk}
    \theta^{pq}\paraa{\d_px^k}\paraa{\d_qx^l}X_l\d_i\\
    &=\theta^{ab}\theta^{pq}g_{bp}\paraa{\d_ax^i}\paraa{\d_qx^l}X_l\d_i,
  \end{align*}
  and by using the fact that $\theta$ is an almost \KP\ structure one
  gets
  \begin{align*}
    \P^2(X) &= -\gamma^2g^{aq}\paraa{\d_ax^i}\paraa{\d_qx^l}X_l\d_i\\
    &=-\gamma^2g^{aq}g\paraa{e_q,X}e_a.
  \end{align*}
  Formula (\ref{eq:TrPsq}) can be proven in a similar way.
\end{proof}

\noindent The above result shows that $\gamma^{-1}\P$ is an extension
of the complex structure of $\TSigma$ to $TM$ such that
$-\gamma^{-2}\P^2$ is the orthogonal projection onto $\TSigma$.  Note
that
\begin{align}
  -\frac{1}{\gamma^2}(\P^2)^{ij} = \D^{ij}
\end{align}
and that the projection $\Pi$, onto $\TSigma^\perp$, can then be written as
\begin{align}\label{eq:Piprojection}
  \Pi^{ij} =
  \sum_{A=1}^pN_A^iN_A^j=\eta^{ij}-\D^{ij}.
\end{align}
For convenience, we shall also consider the map $\D:TM\to TM$, defined
as $\D(X)=\D^i(x^k)X_k\d_i$. It follows from Proposition
\ref{prop:Pproperties} that $\nablah_X$ coincides with $\nablab_X$ for
all $X\in\TSigma$. Namely
\begin{align}
  \nablah_X = X_i\D^i(x^k)\nablab_k = X^k\nablab_k = \nablab_X,
\end{align}
since $X_i\D^i(x^k)=X^k$ for all $X\in\TSigma$. Having at hand the
covariant derivative in the direction of a vector in $\TSigma$,
together with the projection operator, enables us to obtain the
covariant derivative on $\Sigma$ as
\begin{align}\label{eq:nablaXYProjection}
  \nabla_XY = \D\paraa{\nablah_XY},
\end{align}
which can again be written in terms of Poisson brackets. In the same
way, this gives us the second fundamental form as
\begin{align}\label{eq:SFFPb}
  \alpha(X,Y) = \Pi\paraa{\nablah_XY}.
\end{align}
Apart from $\P$, there is another fundamental object given as
\begin{align}
  &\B_A^{ij} = -\gamma^2\nablah^iN_A^j
\end{align}
and one notes that $\B_A$ can be written in terms of Poisson brackets
as
\begin{align*}
  \B_A^{ij} &= \{x^i,x^k\}\eta_{kl}\{x^l,N_A^j\}
  +\{x^i,x^k\}\eta_{kl}\{x^l,x^{m_1}\}\Gammab^j_{m_1m_2}N_A^{m_2}.
\end{align*}

\noindent Just as $\gamma^{-1}\P$ is an extension of the almost
complex structure on $\TSigma$ to $TM$, the map $\gamma^{-2}\B_A$ is
an extension of the Weingarten map.

\begin{proposition}\label{prop:BAprop}
  For $X\in TM$ it holds that
  \begin{align}
    \B_A(X) = -\gamma^2\eta(X,\nablab_{e_a}N_A)g^{ab}e_b,
  \end{align}
  and for $Y\in\TSigma$ one obtains
  \begin{align}
    &\B_A(Y) = \gamma^2W_A(Y),
  \end{align}
  where $W_A$ is the Weingarten map associated to the normal vector
  $N_A$.
\end{proposition}

\begin{proof}
  By using (\ref{eq:Diformula}) one obtains
  \begin{align*}
    \B_A(X) &= -\gamma^2\paraa{\nablah^iN_A^j}X_j\d_i
    = -\gamma^2g^{ab}\paraa{\d_ax^i}\paraa{\d_bx^k}\paraa{\nablab_kN_A^j}X_j\d_i\\
    &=-\gamma^2\eta(X,\nablab_{e_b}N_A)g^{ab}e_a,
  \end{align*}
  and for $Y\in\TSigma$ one gets
  \begin{align*}
    \B_A(Y) &= -\gamma^2Y^c\eta\paraa{e_c,\nablab_{e_b}N_A}g^{ab}e_a
    =\gamma^2Y^ch_{A,cb}g^{ab}e_a\\
    &=\gamma^2(W_A)^a_cY^ce_a = \gamma^2W_A(Y),
  \end{align*}
  from the definition of the Weingarten map.
\end{proof}

\noindent It turns out that the extension of the Weingarten map is
such that the action on normal vectors gives the covariant derivative
in $\TSigma^\perp$.

\begin{proposition}\label{prop:BAnormProp}
  For $X\in\TSigma$ it holds that
  \begin{align}
    \eta\paraa{\B_A(N_B),X} = -\gamma^2\paraa{D_X}_{AB},
  \end{align}
  where $(D_X)_{AB}$ is defined through $D_XN_A=(D_X)_{AB}N_B$.
\end{proposition}

\begin{proof}
  It holds that $(D_X)_{AB} = \eta(D_XN_A,N_B)$, and from Weingarten's
  formula one obtains
  \begin{align*}
     \paraa{D_X}_{AB} = \eta(\nablab_{X}N_A+W_A(Y),N_B)
     =\eta(\nablab_{X}N_A,N_B).
  \end{align*}
  On the other hand, one gets
  \begin{align*}
    \eta\paraa{B_A(N_B),X} &= 
    -\gamma^2\eta(N_B,\nablab_{e_a}N_A)g^{ab}\eta\paraa{e_b,X}
    =-\gamma^2X^c\eta(N_B,\nablab_{e_a}N_A)\delta^a_c\\
    &=-\gamma^2\eta(N_B,\nablab_XN_A)
    =-\gamma^2\paraa{D_X}_{AB},
  \end{align*}
  which proves the statement.
\end{proof}

\noindent Recall Weingarten's formula
\begin{align*}
  \nablab_XN_A = -W_A(X)+D_XN_A
\end{align*}
for all $X\in\TSigma$. Having both $W_A$ and $D_X$ expressed in terms
of $\B_A$ implies that Weingarten's formula gives a nontrivial
relation involving derivatives of normal vectors.
\begin{proposition}\label{prop:WeingartenConsequence}
  For all $X\in\TSigma$ is holds that
  \begin{align}\label{eq:WeingartenConsequence}
    \paraa{\nablah_iN_A^k}X_k = N_A^l\paraa{\nablah^k\Pi_{il}}X_k,
  \end{align}
  from which it follows that
  \begin{align}
    \nablah_kN_A^k = N_A^l\paraa{\nablah^i\Pi_{il}}.
  \end{align}
\end{proposition}
\begin{proof}
  For $X\in\TSigma$ we have previously shown that the following holds
  \begin{align*}
    &\nablab_XN_A^i = \nablah_XN_A^i = \paraa{\nablah^kN_A^i}X_k\\
    &W_A(X)^i = -\paraa{\nablah^iN_A^k}X_k\\
    &D_XN_A^i = X_k(\nablah^kN_A^l)(N_B)_lN_B^i = \Pi^i_l(\nablah^kN_A^l)X_k.
  \end{align*}
  Therefore, Weingarten's formula can be rewritten as
  \begin{align*}
    \paraa{\nablah^kN_A^i}X_k &= 
    \paraa{\nablah^iN_A^k}X_k + \Pi^i_l(\nablah^kN_A^l)X_k\\
    &= \paraa{\nablah^iN_A^k}X_k  + \paraa{\nablah^k\Pi^i_lN_A^l}X_k
    -N_A^l\paraa{\nablah^k\Pi^i_l}X_k\\
    &= \paraa{\nablah^iN_A^k}X_k+\paraa{\nablah^kN_A^i}X_k
    -N_A^l\paraa{\nablah^k\Pi^i_l}X_k,
  \end{align*}
  from which it follows that
  \begin{align*}
    \paraa{\nablah_iN_A^k}X_k = N_A^l\paraa{\nablah^k\Pi_{il}}X_k.
  \end{align*}
  In particular, one may replace $X_k$ by $\D^i_k$, giving
  \begin{align*}
    \nablah_kN_A^k = N_A^l\paraa{\nablah^i\Pi_{il}},
  \end{align*}
  since $\D^i_k\nablah^k = \nablah^i$.
\end{proof}

\noindent Let us know turn to Gauss' equation and the curvature of
$\Sigma$. By using Weingarten's equation and Proposition
\ref{prop:BAprop} one finds the following formulas.

\begin{proposition}\label{prop:etaalphaalpha}
  For $X,Y,Z,V\in\TSigma$ it holds that
  \begin{align*}
    \eta\paraa{\alpha(X,Y),\alpha(Z,V)}
    &= \frac{1}{\gamma^4}\sum_{A=1}^p\eta\paraa{\B_A(X),Y}\eta\paraa{\B_A(Z),V}\\
    &= X_iY_jZ_kV_l\paraa{\nablah^j\Pi^{im}}\paraa{\nablah^l\Pi^k_m}.
  \end{align*}
\end{proposition}

\begin{proof}
  Since, by Proposition \ref{prop:BAprop}, $\B_A$ is proportional to
  $W_A$ when acting on vectors in $\TSigma$, the first equation
  follows directly from Weingarten's equation. 

  By using (\ref{eq:SFFPb}), and the fact that $X,Y,Z,V\in\TSigma$,
  one immediately obtains
  \begin{align*}
    \eta\paraa{\alpha(X,Y),\alpha(Z,V)}
    = X_iY_jZ_kV_l\paraa{\nablah^iN_A^j}\paraa{\nablah^kN_A^l},
  \end{align*}
  and by applying Proposition \ref{prop:WeingartenConsequence} twice
  one arrives at the desired result. 
\end{proof}

\noindent From Proposition \ref{prop:etaalphaalpha} and Gauss' equation
the following result is immediate.

\begin{proposition}\label{prop:curvature}
  Let $\Rb$ and $R$ be the curvature tensors of $M$ and $\Sigma$
  respectively. For $X,Y,Z,V\in\TSigma$ it holds that
  \begin{equation*}
    \begin{split}
      &R(X,Y,Z,V) =
      X^iY^jZ^kV^l\bracketb{\Rb_{ijkl}+
        \paraa{\nablah_k\Pi_{im}}\paraa{\nablah_l\Pi_j^m}
      -\paraa{\nablah_l\Pi_{im}}\paraa{\nablah_k\Pi^m_j}}\\
      &\!\!=\Rb(X,Y,Z,V)
      +\frac{1}{\gamma^4}\sum_{A=1}^p
      \bracketb{\eta\paraa{\B_A(X),Z}\eta\paraa{\B_A(Y),V}
      -\eta\paraa{\B_A(X),V}\eta\paraa{\B_A(Y),Z}}.
    \end{split}
  \end{equation*}
\end{proposition}

\noindent To compute the Ricci curvature and the scalar curvature, one
needs to take the trace over $\TSigma$ of the tensor in Proposition
\ref{prop:curvature}. This can be done by applying projection
operators before tracing over $TM$. That is, the Ricci tensor of
$\Sigma$ can be computed as
$\R_{ik}=R_{ijkl}\D^{jm}\D^{l}_m=R_{ijkl}\D^{jl}$, which implies that
\begin{equation}\label{eq:RicciCurvature}
  \R_{ik}=\D^{jl}\Rb_{ijkl}+\paraa{\nablah_k\Pi_{im}}\paraa{\nablah_l\Pi^{lm}}
  -\paraa{\nablah_l\Pi_{im}}\paraa{\nablah_k\Pi^{lm}}.
\end{equation}
In the same way, the scalar curvature is computed to be
\begin{equation}
  R = \D^{jl}\D^{ik}\Rb_{ijkl}+\paraa{\nablah_k\Pi^{km}}\paraa{\nablah^l\Pi_{lm}}
  -\paraa{\nablah^l\Pi^{km}}\paraa{\nablah_k\Pi_{lm}},
\end{equation}
which (by using Proposition \ref{prop:WeingartenConsequence}) is equal to
\begin{equation}\label{eq:ScalarCurvature}
  R = \D^{jl}\D^{ik}\Rb_{ijkl}+\paraa{\nablah_k\Pi^{km}}\paraa{\nablah^l\Pi_{lm}}
  -\frac{1}{2}\paraa{\nablah^l\Pi^{km}}\paraa{\nablah_l\Pi_{km}}.
\end{equation}

\subsection{The Codazzi-Mainardi equations}

\noindent For submanifolds there are two fundamental sets of equations:
Gauss' equations and the Codazzi-Mainardi equations. Having considered
Gauss' equations in the previous section, let us now turn to the
Codazzi-Mainardi equations. These equations express the normal
component of the curvature in terms of covariant derivatives of the
second fundamental form and covariant derivatives of normal
vectors. That is, for all $X,Y,Z\in\TSigma$ it holds that
\begin{align*}
  \Pi\paraa{\Rb(X,Y)Z} &= 
  \sum_{A=1}^p\bracketb{\paraa{\nabla_Xh_A}(Y,Z)-\paraa{\nabla_Yh_A}(X,Z)}N_A\\
  &\qquad
  +\sum_{A=1}^p\bracketb{h_A(Y,Z)D_XN_A-h_A(X,Z)D_YN_A}.
\end{align*}

\noindent Let us now try to rewrite these equations in terms of
Poisson brackets. 

\begin{proposition}
  For $X,Y,Z\in\TSigma$ it holds that
  \begin{align*}
    (\nabla_Xh_A)(Y,Z)-(\nabla_Yh_A)(X,Z)    
    &=\eta\parab{\paraa{\nablah_X\gamma^{-2}\B_A}(Y)
      -\paraa{\nablah_Y\gamma^{-2}\B_A}(X),Z}\\
    &=X^iY^jZ^k\parab{\nablah_j\nablah_k(N_A)_i-\nablah_i\nablah_k(N_A)_j}.
  \end{align*}
\end{proposition}

\begin{proof}
  Let us start by using Weingarten's equation to rewrite
  \begin{align*}
    \paraa{\nabla_Xh_A}(Y,Z) 
    &= X\cdot h_A(Y,Z)-h_A(\nabla_XY,Z)-h_A(Y,\nabla_XZ)\\
    &= X\cdot g(W_A(Y),Z)-g(W_A(\nabla_XY),Z)-g(W_A(Y),\nabla_XZ)\\
    &= g\paraa{(\nabla_XW_A)(Y),Z}.
  \end{align*}
  On the other hand, one gets
  \begin{align*}
    \eta\paraa{(\nablab_X\gamma^{-2}\B_A)(Y),Z}
    =\eta\paraa{(\nabla_XW_A)(Y),Z}
    -\frac{1}{\gamma^2}\eta\paraa{\B_A(\alpha(X,Y)),Z},
  \end{align*}
  which implies that
  \begin{align*}
    (\nabla_Xh_A)(Y,Z)-(\nabla_Yh_A)(X,Z)    
    &=\eta\parab{\paraa{\nablah_X\gamma^{-2}\B_A}(Y)
      -\paraa{\nablah_Y\gamma^{-2}\B_A}(X),Z}
  \end{align*}
  since $\nablab_X=\nablah_X$ for all $X\in\TSigma$. The second
  formula is obtained by simply inserting the definition of $\B_A$ in
  the above expression.
\end{proof}

\begin{proposition}
  For $X,Y,Z\in\TSigma$ it holds that
  \begin{align*}
    &\sum_{B=1}^p\bracketb{h_B(Y,Z)(D_X)_{BA}-h_B(X,Z)(D_Y)_{BA}}
    =X_iY_jZ_k\bracketb{
      \nablah^k\nablah^iN_A^j-\nablah^k\nablah^jN_A^i
    }.
  \end{align*}
\end{proposition}

\begin{proof}
  Since $h_A(Y,Z)=\eta(\alpha(Y,Z),N_A)$ it follows from
  (\ref{eq:SFFPb}) that 
  \begin{align*}
    h_B(Y,Z) = \Pi^k_iZ_l\paraa{\nablah^lY_k}N_B^i,
  \end{align*}
  and from Proposition \ref{prop:BAnormProp} that
  \begin{align*}
    \paraa{D_X}_{BA} = \paraa{\nablah^iN_B^k}X_i(N_A)_k
    =-\paraa{\nablah^iN_A^k}X_i(N_B)_k.
  \end{align*}
  Thus, one obtains
  \begin{align}
    \sum_{B=1}^ph_B(Y,Z)(D_X)_{BA} &=
    -\Pi^{ik}\Pi_{im}X_jZ_l\paraa{\nablah^lY_k}\paraa{\nablah^jN_A^m}\notag\\
    &=X_jZ_lY_k\paraa{\nablah^l\Pi^k_m}\paraa{\nablah^jN_A^m},\label{eq:eq1}
  \end{align}
  and using the definition of $\Pi^{ij}$ one gets
  \begin{align*}
    X_jZ_lY_k\paraa{\nablah^l\Pi^k_m}\paraa{\nablah^jN_A^m}
    &= X_jZ_lY_k\nablah^l\paraa{\Pi^k_m\nablah^jN_A^m}\\
   &=X_jZ_lY_k\parab{\nablah^l\nablah^jN_A^k-\nablah^l\paraa{\D^k_m\nablah^jN_A^m}}.
  \end{align*}
  Now, let us rewrite the second term using Proposition
  \ref{prop:WeingartenConsequence}.
  \begin{align*}
    X_jZ_lY_k\nablah^l\paraa{\D^k_m\nablah^jN_A^m} &=
    X_jZ_lY_k\nablah^l\paraa{\D^k_mN_A^n\nablah^m\Pi^j_n}=
    X_jZ_lY_k\nablah^l\paraa{N_A^n\nablah^k\Pi^j_n}\\
    &= X_jZ_lY_k\parab{\nablah^l\nablah^kN_A^j
      -\nablah^l\paraa{\Pi^j_n\nablah^kN_A^n}}\\
    &=X_jZ_lY_k\parab{\nablah^l\nablah^kN_A^j
    -\paraa{\nablah^l\Pi^j_n}\paraa{\nablah^kN_A^n}}.
  \end{align*}
  Comparing this with (\ref{eq:eq1}) one obtains
  \begin{align*}
    &\sum_{B=1}^p\bracketb{h_B(Y,Z)(D_X)_{BA}-h_B(X,Z)(D_Y)_{BA}}
    = X_jZ_lY_k\parab{\nablah^l\nablah^jN_A^k-\nablah^l\nablah^kN_A^j},
  \end{align*}
  which is equal to the stated formula.
\end{proof}

\noindent Since
$\eta\paraa{N_A,\Pi(\Rb(X,Y)Z)}=-\Rb_{ijkl}X^iY^jZ^kN_A^l$ one can now
formulate the Codazzi-Mainardi equations in the following way.

\begin{proposition}[The Codazzi-Mainardi equations]\label{prop:CMequations}
  \textrm{ }\\ For all $X,Y,Z\in\TSigma$ it holds that
  \begin{align*}
    X^iY^jZ^kN_A^l\Rb_{ijkl} &=
    X_iY_jZ_k\bracketb{\nablah^i\nablah^kN_A^j - \nablah^j\nablah^kN_A^i
      +\nablah^k\nablah^jN_A^i-\nablah^k\nablah^iN_A^j
    }.
  \end{align*}
\end{proposition}

\subsection{Covariant derivatives and curvature}\label{sec:coderiv}

\noindent As we have seen, the operator $\nablah_X$ coincides with the
covariant derivative in the ambient space $M$, in the direction of a
vector $X\in\TSigma$. Thus, expressions of the type $g(\nabla_XY,Z)$,
where $X,Y,Z\in\TSigma$, can be computed as
$\eta(\nablah_XY,Z)$. In particular, one obtains the following formulas.

\begin{proposition}\label{prop:covderivFormulas}
  Let $\nabla$ denote the covariant derivative on $\Sigma$. For $u,v\in
  C^\infty(\Sigma)$ and $X\in\TSigma$ it holds that
  \begin{align}
    \nabla u &= \nh^i(u)\d_i\label{eq:cfGradient}\\
    \div(X) &= \nablah_iX^i\label{eq:cfDivergence}\\
    \Delta(u)&=\nh_i\nh^i(u)\label{eq:cfLaplacian}\\
    |\nabla^2u|^2&=\nh_i\nh^j(u)\nh_j\nh^i(u)\label{eq:nablaSquSq},
  \end{align}
  where $\div(X)$ is the divergence of $X$ on $\Sigma$, and $\Delta$
  is the Laplace operator on $\Sigma$.
\end{proposition}

\begin{proof}
  Let us prove equation (\ref{eq:cfDivergence}) and equation
  (\ref{eq:cfLaplacian}). The remaining formulas can be proven in an
  analogous way. Since $\D^{ik}=g^{ab}(\d_ax^i)(\d_bx^k)$, Gauss'
  formula gives
  \begin{align*}
    \nablah_iX^i &= \D^{ik}\nablab_iX_k = 
    g^{ab}\paraa{\d_ax^i}\nablab_{e_b}X_k
    =g^{ab}\paraa{\d_ax^i}\parab{(\nabla_bX^c)(e_c)_i+\alpha(e_b,X)_i}\\
    &= g^{ab}g_{ac}\paraa{\nabla_bX^c} = \nabla_bX^b=\div(X).
  \end{align*}
  We prove equation (\ref{eq:cfLaplacian}) by making use of normal
  coordinates on $\Sigma$. Thus, we assume that $u^1,\ldots,u^n$ is a
  set of normal coordinates on $\Sigma$. In particular, this implies
  that $\nabla_a=\d_a$ and $\d_ag_{bc}=0$. With the help of Gauss
  formula, one computes
  \begin{align*}
    \nablah_i\nablah^i(u) &= 
    \eta_{ij}g^{ab}\paraa{\d_ax^i}\nablab_{e_b}
    \parab{g^{pq}\paraa{\d_px^j}\paraa{\d_q u}}\\
    &=\eta_{ij}g^{ab}\paraa{\d_ax^i}
    \nabla_b\parab{g^{pq}\paraa{\d_q u}}\paraa{\d_px^j}\\
    &= g^{ab}g_{ap}g^{pq}\paraa{\d^2_{bq} u} = g^{bq}\paraa{\d^2_{bq} u},
  \end{align*}
  which is equal to $\Delta(u)$ in normal coordinates.
\end{proof}

\noindent Let us investigate how the operator $\nh$ is related to
curvature.  By the very definition of curvature, it arises as the
commutator of two covariant derivatives; is there a similar relation
for $\nh^i$? In fact, when contracted with vectors in $\TSigma$,
$\nablah$ fulfills a curvature equation analogous to the one of
$\nablab$.

\begin{proposition}\label{prop:effectivelySymmetric}
  For any $u\in C^\infty(\Sigma)$, $Z\in TM$ and $X,Y\in\TSigma$
  it holds that
  \begin{align}
    &X^iY^j[\nh_i,\nh_j]Z^k = \Rb^k_{lij}X^iY^jZ^l\label{eq:comNablahZ}\\
    &X^iY^j[\nh_i,\nh_j](u) = 0.\label{eq:comNablahu}
  \end{align}
\end{proposition}

\begin{proof}
  One computes
  \begin{align*}
    X_iY_j&\nablah^i\nablah^jZ^k=
    X_iY_j\D^{il}\nablab_l\parab{\D^{jm}\nablab_mZ^k}\\
    &=X_iY_j\D^{il}\D^{jm}\nablab_l\nablab_mZ^k
    +X_iY_j\D^{il}\paraa{\nablab_l\D^{jm}}\nablab_mZ^k\\
    &=X_iY_j\D^{il}\D^{jm}\nablab_m\nablab_lZ^k
    +X_iY_j\D^{il}\D^{jm}\Rb^k_{nlm}Z^n
    +X_iY_j\D^{il}\paraa{\nablab_l\D^{jm}}\nablab_mZ^k.
  \end{align*}
  Let us now prove that the last term is symmetric in $X$ and
  $Y$. From Gauss' formula it follows that, for $X,Y\in\TSigma$ and
  $V=U+N$, with $U\in\TSigma$ and $N\in\TSigma^\perp$, for any tensor
  of the form $\tilde{T}_{ik}=T^{ab}(e_a)_i(e_b)_k$
  \begin{align*}
    \paraa{\nablab_{X}\tilde{T}}(Y,V)
    = \paraa{\nabla_XT}(Y,U)-\tilde{T}(Y,W_N(X)).
  \end{align*}
  Applying this to the expression above, with $T^{ab}=g^{ab}$ and
  $U_m=\nabla_mZ^k$, gives 
  \begin{align*}
    X_iY_j\D^{il}\paraa{\nablab_l\D^{jm}}\nablab_mZ^k
    &= \paraa{\nablab_{X}\tilde{T}}\paraa{Y,(\nablab^mZ^k)\d_m}
    = -\tilde{T}\paraa{Y,W_N(X)}\\
    &=g^{ab}(e_a)_i(e_b)_kY^iW_N(X)^k
    =g(Y,W_N(X))\\
    &= h_N(X,Y),
  \end{align*}
  where the last equality is Weingarten's equation. Thus, the
  expression is symmetric since the second fundamental form is
  symmetric. Hence, one obtains
  \begin{align*}
    X_iY_j\nablah^i\nablah^jZ^k&=
    X_iY_j\nablah^j\nablah^iZ^k+X_iY_j\D^{il}\D^{jm}\Rb^k_{nlm}Z^n\\ 
    &\quad+X_iY_j\D^{il}\paraa{\nablab_l\D^{jm}}\nablab_mZ^k
    -X_iY_j\D^{jm}\paraa{\nablab_m\D^{il}}\nablab_lZ^k\\    
    &=X_iY_j\nablah^j\nablah^iZ^k+X^lY^m\Rb^k_{nlm}Z^n.
  \end{align*}
  Equation (\ref{eq:comNablahu}) is proven in an analogous way.
\end{proof}

\noindent Let us illustrate that the operator $\nablah$ is also related
to the curvature on $\Sigma$.  Namely, consider the following equation
\begin{equation}\label{eq:curvatureEq}
  (\nabla^a u)\nabla_a\nabla_b\nabla^b u=(\nabla^a u)\nabla_b
  \nabla_a\nabla^b u -\R(\nabla u,\nabla u),
\end{equation}
where $\R$ is the Ricci curvature of $\Sigma$, which is a particular
instance of the relation between curvature and covariant derivatives
on $\Sigma$. Let us rewrite this equation in terms of $\nh^i$. From
Proposition \ref{prop:covderivFormulas} it immediately follows that
\begin{align*}
  (\nabla^a u)\nabla_a\nabla_b\nabla^b u = 
  \nh^i(u)\nh_i\nh_j\nh^j(u),
\end{align*}
and from 
\begin{equation}
  \Delta\paraa{|\nabla u|^2} = 
  2\paraa{\nabla^a u}\nabla_b \nabla_a \nabla^b  u+2|\nabla^2u|^2,
\end{equation}
one obtains 
\begin{align*}
  (\nabla^a u)\nabla_b\nabla_a\nabla^b u&=
  \frac{1}{2}\nh_i\nh^i\paraa{\nh^j(u)\nh_j(u)}-\nh_i\nh^j(u)\nh_j\nh^i(u)\\
  &=\nh_i\nh^i\nh^j(u)\nh_j(u)+[\nh_i,\nh^j](u)\nh^i\nh_j(u),
\end{align*}
by again using Proposition \ref{prop:covderivFormulas}. Thus, we can
write eq. (\ref{eq:curvatureEq}) as
\begin{equation}\label{eq:graduRicci}
  \begin{split}
    \nh^i(u)\nh_i\nh_j\nh^j(u) &= \nh_i\nh^i\nh^j(u)\nh_j(u)+
    [\nh_i,\nh^j](u)\nh^i\nh_j(u)\\
    &\quad-\R(\nabla u,\nabla u).
  \end{split}
\end{equation}
Turning this equation around, it gives an expression for the Ricci
curvature evaluated at $\nabla u$. Can the same expression be directly
derived from equation (\ref{eq:RicciCurvature})? Let us now
prove that it is indeed possible. 

\begin{proposition}\label{prop:rewriteRicci}
  For all $u\in C^\infty(\Sigma)$ it holds that
  \begin{align*}
    \paraa{\nablah_k\Pi_{im}}\paraa{\nablah_l\Pi^{lm}}\nablah^i(u)\nablah^k(u)
    &=\nablah_l\parab{\nablah^k(u)\nablah^l\nablah_k(u)}-
    \nablah^i\parab{\nablah^k(u)\nablah_k\nablah_i(u)}\\
    -\paraa{\nablah_l\Pi_{im}}\paraa{\nablah_k\Pi^{lm}}\nablah^i(u)\nablah^k(u)
    &=-\nablah^k(u)\nablah_k\nablah^i\nablah_i(u)
    +\nablah^k(u)\nablah^i\nablah_k\nablah_i(u)\\
    &\qquad-\D^{jl}\Rb_{ijkl}\nablah^k(u)\nablah^i(u).
  \end{align*}
\end{proposition}

\begin{proof}
  Let us prove the second formula. One computes
  \begin{align*}
    -&\paraa{\nablah_l\Pi_{im}}\paraa{\nablah_k\Pi^{lm}}\nablah^i(u)\nablah^k(u)
    =\Pi_{im}\nablah_l\nablah^i(u)\paraa{\nablah_k\Pi^{lm}}\nablah^k(u)\\
    &= \nablah_l\nablah_i(u)\nablah^k(u)\paraa{\nablah_k\Pi^{il}}
    =-\nablah_l\nablah_i(u)\nablah^k(u)\paraa{\nablah_k\D^{il}}\\
    &= -\nablah^k(u)\nablah_k\parab{\D^{il}\nablah_l\nablah_i(u)}
    +\nablah^k(u)\D^{il}\nablah_k\nablah_l\nablah_i(u),
  \end{align*}
  and by applying Proposition \ref{prop:effectivelySymmetric} to
  $[\nablah_k,\nablah_l]\nablah_i(u)$ a curvature term appears and one
  obtains the stated formula.
\end{proof}

\noindent Applying the above result to $\R(\nabla u,\nabla u)$, by
using formula (\ref{eq:RicciCurvature}) for the Ricci curvature, one
reproduces equation (\ref{eq:graduRicci}).

\subsection{Integrable structures and K\"ahler manifolds}\label{sec:integration}

\noindent Let us investigate the important case when $\Sigma$ is a
K\"ahler manifold with respect to $g$ and $\J=\gamma^{-1}\theta$.  In
particular, this implies that the complex structure, which is now
integrable, is parallel with respect to the Levi-Civita connection,
which allows for a simplification of several formulas in Section
\ref{sec:coderiv}. The reason for never having to consider the
derivative of the Poisson bivector so far, is that everything was
expressed in terms of $\D^i(u)$, which in local coordinates becomes
\begin{align*}
  \D^i(u) = g^{ab}(\d_ax^i)(\d_bu),
\end{align*}
due to the fact that $\theta$ is an almost \KP\ structure. Thus, there
are no explicit dependencies on $\theta$ left, and any derivative
acting on $\D^i(u)$ will only produce derivatives of the metric. 

For K\"ahler manifolds, one need not worry about derivatives of
$\gamma^{-1}\theta$, since the complex structure is covariantly
constant, and instead of $\nablah$, one may consider
$\nablat$, defined by
\begin{align*}
  \nablat^i = \Dt^{ik}\nablab_k\equiv \frac{1}{\gamma}\{x^i,x^k\}\nablab_k.
\end{align*}
Recall from Proposition \ref{prop:Pproperties} that $\Dt$ is in fact the
associated complex structure. Therefore, the equation $\nabla\J=0$
can be formulated as follows.

\begin{proposition}\label{prop:KPstructureConsequence}
  Assume that $(\Sigma,g,\J)$ is a K\"ahler manifold. For any
  $X,Y\in\TSigma$ it holds that
  \begin{align}
    X_jY_k\paraa{\nablat^i\Dt^{jk}} = 0.
  \end{align}
\end{proposition}

\begin{proof}
  Let us assume that $u^a$ is a set of normal coordinates on
  $\Sigma$. One obtains
  \begin{align*}
    X_jY_k\paraa{\nablat^i\Dt^{jk}} &=
    \frac{1}{\gamma}X_jY_k\{x^i,x^l\}\nablab_l\frac{1}{\gamma}\{x^j,x^k\}\\
    &=\frac{1}{\gamma}X_jY_k\theta^{ab}(\d_ax^i)\nablab_{e_b}
    \frac{1}{\gamma}\theta^{pq}(\d_px^j)(\d_qx^k)\\
    &=\frac{1}{\gamma}X_jY_k\theta^{ab}(\d_ax^i)\nabla_{e_b}
    \frac{1}{\gamma}\theta^{pq}(\d_px^j)(\d_qx^k)
  \end{align*}
  since $\nablab_{e_b}$ coincides with $\nabla_{e_b}$ when contracted
  with vectors in $\TSigma$. Now, we use the fact that
  $\nabla\gamma^{-1}\theta=0$ and normal coordinates to obtain
  \begin{align*}
    X_jY_k\paraa{\nablat^i\Dt^{jk}}
    &=\frac{1}{\gamma^2}X_jY_k\theta^{ab}\theta^{pq}(\d_ax^i)\d_b
    \paraa{(\d_px^j)(\d_qx^k)} = 0,
  \end{align*}
  since $X_j\d^2_{ab}x^j = X^c(\d_cx^i)\eta_{ij}\d^2_{ab}x^j=0$ in
  normal coordinates.
\end{proof}

\noindent Since $\nablah^i = -\Dt^{ik}\nablat_k$, it is easy to see
why one is allowed to replace $\nablah$ by $\nablat$ in many of the
formulas in Section \ref{sec:coderiv}. For instance
\begin{align*}
  \nablah^i\nablah_i(u) = \Dt^{ik}\nablat_k\paraa{\Dt_{il}\nablat^l(u)}
  =\D^{kl}\nablat_k\nablat_l(u)+\Dt^{ik}\paraa{\nablat_k\Dt_{il}}\nablat^l(u)
  =\nablat^l\nablat_l(u).
\end{align*}

\noindent Let us end this section with a couple of words about
integration. Assume that $\Sigma$ is a closed manifold, in which case
Stoke's theorem tells us that
\begin{align}\label{eq:StokesFormula}
  \int_\Sigma\div(X) = 0
\end{align}
for all $X\in\TSigma$. If $\Sigma$ carries an almost \KP\ structure,
it follows from Proposition \ref{prop:covderivFormulas} that
\begin{align*}
  \int_\Sigma\nablah_iX^i = 0,
\end{align*}
which implies that the standard rule for partial integration also
holds for $\nablah$. In case $\Sigma$ is a K\"ahler manifold, a
similar formula holds for $\nablat$, since one can show that
$\nablat_iX^i=-\nablah^i\paraa{\Dt_{ik}X^k}$.

\section{\KP\ algebras}\label{sec:KPalgebras}

\noindent In this section we shall consider the algebraic version of
the previous results, and find an intrinsic definition of a
Poisson algebra that corresponds to a (complex) function algebra on a
submanifold of $\reals^m$. Thus, we consider Poisson algebras $\A$
(with a unit) over $\complex$ generated (as algebras) by $m$ elements
$x^1,\ldots,x^m$, for which we denote
\begin{align}
  \P^{ij} = \{x^i,x^j\}.
\end{align}
It is interesting to carry through the constructions that will follow
below, when $\A$ is an arbitrary algebra generated by $x^1,\ldots,x^m$;
however, as the goal of the current exposition is rather to explore
what kind of results that can be obtained in the algebraic setting, we
shall at this stage avoid unnecessary complications by considering
$\A$ to be the field of fractions of the polynomial ring
$\complex[x^1,\ldots,x^m]$. This is analogous to considering Poisson
brackets on $\reals^m$ that restrict to functions on a subspace (which
is identified with the submanifold $\Sigma$). For instance (see also
Section \ref{sec:examples}), for an arbitrary polynomial
$C(x^1,x^2,x^3)$ one may define a Poisson bracket of functions on
$\reals^3$ by setting $\{f,g\}=\eps^{ijk}(\d_if)(\d_jg)(\d_kC)$. Since
$\{f,C\}=0$ for all $f$, the Poisson bracket restricts to the quotient
algebra $\complex[x^1,x^2,x^3]/(C)$, which may be identified with the
polynomial functions on the level set
$\Sigma=\{(x^1,x^2,x^3):C(x^1,x^2,x^3)=0\}$.

We let $\A$ have the structure of a $\ast$-algebra by setting
$(x^i)^\ast = x^i$, and we assume that the Poisson structure is such
that $\{u,v\}^\ast = \{u^\ast,v^\ast\}$ for all $u,v\in\A$.  Although
the position of indices (upper or lower) will not matter in what
follows (as the ambient manifold is thought of as $\reals^m$), we
shall keep the notation from differential geometry and also assume
that all repeated indices are summed over from $1$ to $m$.

Let $\Der(\A)$ denote the module of derivations generated by
$\d_i\equiv\d_{x^i}$ for $i=1,\ldots,m$. We equip this module with a
bilinear form $(\cdot,\cdot)$ defined through
\begin{align}\label{eq:defScalarProd}
  (X,Y)\equiv(X^i\d_i,Y^j\d_j) = X^i Y_i,
\end{align}
and we extend the involution to $\Der(\A)$ by setting $X^\ast =
(X^i)^\ast\d_i$. Furthermore, we define 
$\P:\Der(\A)\to\Der(\A)$ through
\begin{align}
  \P(X) = {\P^i}_{\!j}X^j\d_i. 
\end{align}

\noindent In terms of the Poisson tensor $\P^{ij}$, the defining
relation for an almost \KP\ structure can be formulated as
$\P^3(X)=-\gamma^2\P(X)$. We shall take this as a definition for
almost \KP\ algebras.

\begin{definition}\label{def:almostKPalgebra}
  Let $\A$ be the field of fractions of $\complex[x^1,\ldots,x^m]$,
  and let $\{\cdot,\cdot\}$ be a Poisson structure on $\A$, for which
  we set $\P^{ij}=\{x^i,x^j\}$. The Poisson algebra
  $(\A,\{\cdot,\cdot\})$ is called \emph{an almost \KP\ algebra} if
  there exists\footnote{Let us leave the square of $\gamma^2$ in order
    to keep the analogy with differential geometry; however, we shall
    in general not assume that there exists a square root of
    $\gamma^2$.}  an invertible hermitian $\gamma^2\in\A$ such that
  \begin{align}\label{eq:defAlmostKP}
    {\P^i}_k{\P^k}_l{\P^l}_j = -\gamma^2{\P^i}_j.
  \end{align}
\end{definition}

\noindent We shall introduce the notation of Section \ref{sec:diffgeoKP} and
define for $u\in\A$ and $X\in\Der(\A)$
\begin{align*}
  &\D^i(u) = \frac{1}{\gamma^2}\{u,x^k\}{\P^{i}}_{\!k}
  \qquad \D^{ik}=\D^i(x^k)\\
  &\Pi^{ik} = \delta^{ik}-\D^{ik}\qquad\nablah^iX^k = \D^i(X^k),
\end{align*}
as well as $\D(X) = \D^i_kX^k\d_i$, $\D(X,Y)=(\D(X),Y)$ and
$\Pi(X)=\Pi^{i}_kX^k\d_i$. Note that the defining relation
(\ref{eq:defAlmostKP}) implies that $\D$ is a projector on $\Der(\A)$,
i.e. $\D^2(X) = \D(X)$ for all $X\in\Der(\A)$. Therefore, the concepts
of tangent space and normal space arise naturally as projections of
$\Der(\A)$.
\begin{definition}
  Let $\A$ be an almost \KP\ algebra. The \emph{tangent
    space}\footnote{Strictly speaking, $\X(\A)$ should rather be
    called the tangent bundle, since we are always considering
    analogues of global objects on a manifold.}  $\X(\A)$ is defined
  as
  \begin{align}
    \X(\A) = \{\D(X):X\in\Der(\A)\},
  \end{align}
  and the \emph{normal space} $\N(\A)$ is defined as
  \begin{align}
    \N(\A) = \{\Pi(X):X\in\Der(\A)\}.
  \end{align}
\end{definition}

\noindent Clearly, it holds that $\Der(\A)=\X(\A)\oplus\N(\A)$,
$\D(X)=X$ for all $X\in\X(\A)$, and $(X,N)=0$ for all $X\in\X(\A)$ and
$N\in\N(\A)$. In the current situation, all modules are vector
spaces, since $\A$ is a field. Note however that the above setup
allows for an extension to more general rings.

It follows from (\ref{eq:defAlmostKP}) that
\begin{align}
  \gamma^2 = -(\P^4)^i_i/(\P^2)^k_k,
\end{align}
and in case there exists a basis $X_1,\ldots,X_n,Y_1,\ldots,Y_p$ of
$\Der(\A)$ such that $X_1,\ldots,X_n$ is a basis for $\X(\A)$ and
$Y_1,\ldots,Y_p$ is a basis of $\N(\A)$, it is easy to see that
\begin{align}
  \gamma^2 = -\frac{1}{n}\Tr\P^2
\end{align}
where $\Tr\P^2=(\P^2)^i_i$. The dimension of $\X(\A)$, i.e. $n$, is
called the \emph{geometric dimension of $\A$}.

In the differential geometric setting, the special case of a K\"ahler
manifold was studied. How may one proceed to define a \KP\ algebra?
Let us start from Proposition \ref{prop:KPstructureConsequence}, which
states that the following relation holds
\begin{align*}
  X_jY_k\paraa{\nablat^i\Dt^{jk}} = 0,
\end{align*}
where $\Dt^{jk}=\frac{1}{\gamma}\{x^j,x^k\}$ and
$\nablat^i=\Dt^{ik}\nablab_k$. In the case of $\reals^m$, this
statement can be written as
\begin{align}\label{eq:algebraCSparallell}
  \frac{1}{\gamma}\{x^i,\frac{1}{\gamma}\{x^j,x^k\}\}\P_{jl}\P_{km}=0.
\end{align}
However, this expression depends on the existence of a square root of
$\gamma^2$, which in general does not exist. Let us expand the above
expression as
\begin{align*}
  \frac{1}{\gamma}\{x^i,\frac{1}{\gamma}\{x^j,x^k\}\}\P_{jl}\P_{km}
  &=\frac{1}{\gamma^2}\{x^i,\{x^j,x^k\}\}\P_{jl}\P_{km}
  -\frac{1}{2\gamma^4}\{x^i,\gamma^2\}\{x^j,x^k\}\P_{jl}\P_{km}\\
  &=\frac{1}{\gamma^2}\{x^i,\{x^j,x^k\}\}\P_{jl}\P_{km}
  -\frac{1}{2\gamma^2}\{x^i,\gamma^2\}\P_{lm}.
\end{align*}
Thus, one is lead to the following additional requirement on an almost
\KP\ algebra
\begin{align*}
  \{x^i,\{x^j,x^k\}\}\P_{jl}\P_{km} = \frac{1}{2}\{x^i,\gamma^2\}\P_{lm}.
\end{align*}
However, in the following we shall focus on the general case of almost
\KP\ algebras.

\subsection{Curvature}

\noindent Since the formulas for curvature in Section
\ref{sec:diffgeoKP} are expressed in terms of Poisson brackets, it is
natural to introduce curvature in almost \KP\ algebras.  It was shown (see
eq. (\ref{eq:nablaXYProjection})) that the covariant derivative on
$\TSigma$ can be computed as

\begin{align*}
  \nabla_XY = \D\paraa{\nablah_XY}
\end{align*}
for all $X,Y\in\TSigma$. Let us make this into a definition in the
algebraic setting.
\begin{definition}
  Let $\A$ be an almost \KP\ algebra. For any $X\in\X(\A)$, the
  \emph{covariant derivative $\nabla_X:\X(\A)\to\X(\A)$} is
  defined as
  \begin{align}\label{eq:nablaDefinition}
    \nabla_XY = \D^{ik}X^l\D_l(Y_k)\d_i,
  \end{align}
  and in components we shall also write $\nabla_kY^i =
  \D^{il}\D_k(Y_l)$.  Furthermore, for $u\in\A$ we set
  $\nabla_X(u) = X^i\D_i(u)$.
\end{definition}

\begin{proposition}\label{prop:nablaProperties}
  Let $\A$ be an almost \KP\ algebra. For all $X,Y,Z\in\X(\A)$ and  
  $u\in\A$, the covariant derivative has the following
  properties
  \begin{enumerate}
  \item $\nabla_X(Y+Z)=\nabla_XY+\nabla_XZ$,
  \item $\nabla_{(X+Y)}Z=\nabla_XZ+\nabla_YZ$,
  \item $\nabla_{(uX)}Y=u\nabla_XY$,
  \item $\nabla_X(uY)=\nabla_X(u)Y+u\nabla_XY$.
  \end{enumerate}
\end{proposition}

\begin{proof}
  The first three properties are immediate from the definition. Let us
  prove the last one. One computes
  \begin{align*}
    \nabla_X(uY^i) &= \D^{ik}X^l\D_l(uY_k) = 
    u\D^{ik}X^l\D_l(Y_k)+\D^{ik}Y_kX^l\D_l(u)\\
    &=u\nabla_XY^i+Y^i\nabla_X(u),
  \end{align*}
  since $Y\in\X(\A)$.
\end{proof}

\noindent The above result shows that $\nabla$ has all the properties
one expects from an affine connection.  We shall also extend the
action of the covariant derivative to tensors in a standard manner;
for instance
\begin{align*}
  \paraa{\nabla_XT}(Y,Z) = \nabla_XT(Y,Z)-T\paraa{\nabla_XY,Z}
  -T\paraa{Y,\nabla_XZ}.
\end{align*}
The following lemma is important proving further properties of the
covariant derivative and its associated curvature. It is an algebraic
analogue of the fact that in a Riemannian manifold $M$, it holds that
$\nabla_i\nabla_j(u)=\nabla_j\nabla_i(u)$ for $u\in C^\infty(M)$.

\begin{lemma}\label{lemma:DDSymmetric}
  Let $\A$ be an almost \KP\ algebra. For all $u\in\A$ it holds that
  \begin{align*}
    [\D^i,\D^j](u)\P_{ik}\P_{jl}=0.
  \end{align*}
\end{lemma}

\begin{proof}
  One computes
  \begin{align*}
    \D^i\D^j(u)\P_{ik}\P_{jl} &=
    \frac{1}{\gamma^2}\{\D^j(u),x_m\}\P^{im}\P_{ik}\P_{jl}
    =\{\D^j(u),x^m\}\D_{mk}\P_{jl}\\
    &=\{\D^j(u),x_k\}\P_{jl}
    =\{\D^j(u)\P_{jl},x_k\}-\{\P_{jl},x_k\}\D^j(u)\\
    &=\{\{u,x_l\},x_k\}-\{\P_{jl},x_k\}\D^j(u).
  \end{align*}
  Using the Jacobi identity in both terms gives
  \begin{align*}
    \D^i\D^j(u)\P_{ik}\P_{jl} &=
    -\{\{x_l,x_k\},u\}-\{\{x_k,u\},x_l\}\\
    &\quad+\{\P_{lk},x_j\}\D^j(u)+\{\P_{kj},x_l\}\D^j(u)\\
    &=\{\{u,x_k\},x_l\}-\{\P_{jk},x_l\}\D^j(u)\\
    &\quad-\{\{x_l,x_k\},u\}+\{\P_{lk},u\}\\
    &=\{\{u,x_k\},x_l\}-\{\P_{jk},x_l\}\D^j(u)=\D^j\D^i(u)\P_{ik}\P_{jl},
  \end{align*}
  by comparing with the result of the previous computation.
\end{proof}

\noindent The action of on $\A$ of an element $X\in\X(\A)$, has
previously been defined as $X(u)=X^i\D_i(u)$. By setting
$[X,Y]^i=X(Y^i)-Y(X^i)$ one can show that $[X,Y]$ is again an element
of $\X(\A)$ and that the connection is torsion free with respect to
this commutator.

\begin{proposition}
  If $X,Y\in\X(\A)$ then it follows that $[X,Y]\in\X(\A)$. Moreover,
  the covariant derivative in an almost \KP\ algebra has no torsion,
  i.e.
  \begin{align*}
    \nabla_XY-\nabla_YX-[X,Y] = 0
  \end{align*}
  for all $X,Y\in\X(\A)$.
\end{proposition}

\begin{proof}
  One computes
  \begin{align*}
    \D([X,Y])^i &= \D^{ik}[X,Y]_k =
    \D^{ik}X^l\D_l(Y_k)-\D^{ik}Y^l\D_l(X_k)\\
    &=X^l\D_l(Y^i)-Y_kX^l\D_l(\D^{ik})-Y^l\D_l(X_k)+X_kY^l\D_l(\D^{ik}) \\
    &=[X,Y]^i-Y_kX^l\D_l(\D^{ik})+X_kY^l\D_l(\D^{ik}).
  \end{align*}
  The last two terms cancel by Lemma \ref{lemma:DDSymmetric}, which
  shows that $\D([X,Y])=[X,Y]$. Let us now show that the connection is
  torsion free. Writing
  \begin{align*}
    \nabla_XY^i-\nabla_YX^i &= 
    \D^{ik}X^l\D_l(Y_k)-\D^{ik}Y^l\D_l(X_k)= \D([X,Y])^i,
  \end{align*}
  it follows from the previous calculation that this equals
  $[X,Y]^i$.
\end{proof}



\noindent Furthermore, one can show that the connection $\nabla$ is a
metric connection.

\begin{proposition}\label{prop:metricConnection}
  In an almost \KP\ algebra with covariant derivative $\nabla$ it
  holds that
  \begin{align*}
    \nabla_X(Y,Z)-(\nabla_XY,Z)-(Y,\nabla_XZ) = 0
  \end{align*}
  for all $X,Y,Z\in\X(\A)$.
\end{proposition}

\begin{proof}
  One computes
  \begin{align*}
    \nabla_X(Y,Z)-&(\nabla_XY,Z)-(Y,\nabla_XZ)\\
    &=X^m\D_m(Y_iZ^i)-\D^{ik}X^m\D_m(Y_k)Z_i-Y_i\D^{ik}X^m\D_m(Z_k)\\
    &=X^m\D_m(Y_iZ^i)-X^m\D_m(Y_k)Z^k-Y^kX^m\D_m(Z_k) = 0,
  \end{align*}
  which shows that the connection is metric.
\end{proof}

\noindent Let us now proceed and define curvature in the usual
manner.

\begin{definition}
  Let $\A$ be an almost \KP\ algebra and let $\nabla$ be the covariant
  derivative of $\A$. For $X,Y,Z\in\X(\A)$ we define the
  \emph{curvature tensor $R$} via
  \begin{align}
    R(X,Y)Z = \nabla_X\nabla_YZ - \nabla_Y\nabla_X Z - \nabla_{[X,Y]}Z,
  \end{align}
  and we write $R(X,Y,Z,V)=\paraa{R(Z,V)Y,X}$ as well as
  $R(X,Y,Z)=R(X,Y)Z$ for $X,Y,Z,V\in\X(\A)$.
\end{definition}

\noindent We continue by proving the Bianchi identities.

\begin{proposition}\label{prop:bianchi}
  Let $\A$ be an almost \KP\ algebra and let $R$ be the curvature
  tensor of $\A$. For all $X,Y,Z,V\in\X(\A)$ it holds that  
  \begin{align}
    &R(X,Y,Z)+R(Z,X,Y)+R(Y,Z,X) = 0\label{eq:firstBianchi}\\
    &\paraa{\nabla_XR}(Y,Z,V)+\paraa{\nabla_YR}(Z,X,V)+
    \paraa{\nabla_ZR}(X,Y,V)=0\label{eq:secondBianchi}.
  \end{align}
\end{proposition}

\begin{proof}
  The first Bianchi identity (\ref{eq:firstBianchi}) is proven by
  acting with $\nabla_Z$ on the torsion free condition
  $\nabla_XY-\nabla_YX-[X,Y]=0$, and then summing over cyclic
  permutations of $X,Y,Z$. Since $[[X,Y],Z]+[[Y,Z],X]+[[Z,X],Y]=0$,
  the desired result follows.  The second identity is obtained by a
  cyclic permutation (in $X,Y,Z$) of
  $R\paraa{\nabla_XY-\nabla_YX-[X,Y],Z,V}=0$. One has
  \begin{align*}
    0&=R\paraa{\nabla_XY-\nabla_YX-[X,Y],Z,V}+\text{cycl.}\\
    &=R(\nabla_ZX,Y,V)+R(X,\nabla_Z Y,V)-R([X,Y],Z,V) +\text{cycl.}
  \end{align*}
  On the other hand, one has
  \begin{align*}
    (\nabla_ZR)(X,Y,V) &= \nabla_ZR(X,Y,V) -R(\nabla_Z X,Y,V)\\
    &\qquad-R(X,\nabla_ZY,V)
    -R(X,Y,\nabla_ZV),
  \end{align*}
  and substituting this into the previous equation yields
  \begin{align*}
    0 = \nabla_ZR(X,Y,V)-\paraa{\nabla_ZR}(X,Y,V)-R(X,Y,\nabla_ZV)
    -R([X,Y],Z,V)+\text{cycl.}
  \end{align*}
  After inserting the definition of $R$, and using that
  $[[X,Y],Z]+\text{cycl.}=0$, the second Bianchi identity follows.
\end{proof}

\noindent The following proposition shows that the usual symmetries of
the curvature tensor also hold in the algebraic setting.

\begin{proposition}
  The curvature $R$ of an almost \KP\ algebra has the following
  properties
  \begin{align}
    &R(X,Y,Z,V) = -R(X,Y,V,Z) = -R(Y,X,Z,V)\label{eq:Rsymmetry1}\\
    &R(X,Y,Z,V) = R(Z,V,X,Y)\label{eq:Rsymmetry2}
  \end{align}
  for all $X,Y,Z,V\in\X(\A)$.
\end{proposition}
  
\begin{proof}
  The first identity, $R(X,Y,Z,V)=-R(X,Y,V,Z)$ follows directly from
  the definition of $R$. Let us now prove that
  $R(X,Y,Z,V)=-R(Y,X,Z,V)$. It is easy to show that for any $u\in\A$
  it holds that
  $\nabla_X\nabla_Y(u)-\nabla_Y\nabla_X(u)-\nabla_{[X,Y]}(u)=0$. By setting
  $u=(X,Y)$ one obtains
  \begin{align*}
    \nabla_Z\nabla_V(X,Y)-\nabla_Z\nabla_V(X,Y)-\nabla_{[Z,V]}(X,Y)=0,
  \end{align*}
  and since $\nabla$ is a metric connection, it follows that
  \begin{align*}
    &\nabla_Z\bracketb{(\nabla_VX,Y)+(X,\nabla_VY)}
    -\nabla_V\bracketb{(\nabla_ZX,Y)+(X,\nabla_ZY)}\\
    &\qquad
    -(\nabla_{[Z,V]}X,Y)-(X,\nabla_{[Z,V]}Y)=0.
  \end{align*}
  A further expansion of the derivatives yields
  \begin{align*}
    &(\nabla_Z\nabla_VX,Y)+(X,\nabla_Z\nabla_VY)
    -(\nabla_V\nabla_Z X,Y)-(X,\nabla_V\nabla_ZY)\\
    &\qquad
    -(\nabla_{[Z,V]}X,Y)-(X,\nabla_{[Z,V]}Y)=0,
  \end{align*}
  which is equivalent to
  \begin{align*}
    (R(Z,V)X,Y) = -(R(Z,V)Y,X).
  \end{align*}
  It is a standard algebraic result that any quadri-linear map
  satisfying (\ref{eq:Rsymmetry1}) and (\ref{eq:firstBianchi}) also
  satisfies (\ref{eq:Rsymmetry2}) \cite{kn:foundationsDiffGeometryI}.
\end{proof}

\noindent One can now derive Gauss' formula, in the form of Proposition
\ref{prop:curvature}.

\begin{proposition}\label{prop:GaussFormula}
  Let $\A$ be an almost \KP\ algebra with curvature tensor $R$. For
  $X,Y,Z\in\X(\A)$ it holds that
  \begin{align*}
    R(X,Y)Z^i = X^kY^lZ^j\D^{in}\bracketb{
      \paraa{\D_k\Pi_n^m}\paraa{\D_l\Pi_{jm}}
      -\paraa{\D_l\Pi_n^m}\paraa{\D_k\Pi_{jm}}
    }.
  \end{align*}
\end{proposition}

\begin{proof}
  One computes that
  \begin{align*}
    &(\nabla_X\nabla_YZ)^i =
    \D^{ik}X^l\D_l\paraa{\D_{km}Y^n\D_n(Z^m)}\\
    &\quad= \D^{ik}X^l\D_l\paraa{\D_{km}}Y^n\D_n(Z^m)
    +\D^{i}_mX^l\D_l(Y^n)\D_n(Z^m)
    +\D^{i}_mX^lY^n\D_l\D_n(Z^m).
  \end{align*}
  Now, one uses Lemma \ref{lemma:DDSymmetric} in the last term to
  obtain
  \begin{align*}
    \D^i_mX^lY^n\D_l\D_n(Z^m) &=
    \D^i_mX^lY^n\D_n\D_l(Z^m)
    = \D^{ik}\D_{km}X^lY^n\D_n\D_l(Z^m)\\
    &= \D^{ik}Y^n\D_n\paraa{\D_{km}X^l\D_l(Z^m)}
    -\D^{ik}Y^n\D_n(\D_{km}X^l)\D_l(Z^m)\\
    &= \paraa{\nabla_Y\nabla_XZ}^i-\D^{ik}Y^n\D_n(\D_{km}X^l)\D_l(Z^m),
  \end{align*}
  which implies that
  \begin{align*}
    [\nabla_X,\nabla_Y]Z^i &= \D^{ik}X^l\D_l\paraa{\D_{km}}Y^n\D_n(Z^m)
    -\D^{ik}Y^n\D_n(\D_{km})X^l\D_l(Z^m)\\
    &\qquad+\D^{i}_mX^l\D_l(Y^n)\D_n(Z^m)
    -\D^{i}_mY^n\D_n(X^l)\D_l(Z^m)
  \end{align*}
  One easily checks that the two last terms equal $\nabla_{[X,Y]}Z^i$,
  and therefore it holds that
  \begin{align*}
    R(X,Y)Z^i = \D^{ik}X^l\D_l\paraa{\D_{km}}Y^n\D_n(Z^m)
    -\D^{ik}Y^n\D_n(\D_{km})X^l\D_l(Z^m)
  \end{align*}
  Let us consider the first of these two terms (as the other one is
  obtained by interchanging $X$ and $Y$)
  \begin{align*}
    &\D^{ik}X^l\D_l\paraa{\D_{km}}Y^n\D_n(Z^m) =
    -\D^{ik}X^l\D_l(\Pi_{km})Y^n\D_n(Z^m)\\
    &\qquad=-\D^{ik}X^l\D_l(\Pi_{kj}\Pi^j_m)Y^n\D_n(Z^m)
    =-\D^{ik}X^l\Pi^j_m\D_l(\Pi_{kj})Y^n\D_n(Z^m)\\
    &\qquad=\D^{ik}X^lZ^m\D_l(\Pi_{kj})Y^n\D_n(\Pi^j_m)
    =X^lY^nZ^m\parab{\D^{ik}\D_l(\Pi_{kj})\D_n(\Pi^j_m)},
  \end{align*}
  and by inserting this in the previous expression one obtains the
  stated formula.
\end{proof}

\noindent It is convenient to also develop an index notation for the
covariant derivative. Hence, we extend the definition from vectors to
tensors through
\begin{align}\label{eq:covariantTensorDef}
  \nabla_iT^{k_1\cdots k_N}_{l_1\cdots l_M}
  =\D^{k_1}_{k_1'}\cdots\D^{k_N}_{k_N'}
  \D^{l_1'}_{l_1}\cdots\D^{l_N'}_{l_N}
  \D_i\parab{T^{k_1'\cdots k_N'}_{l_1'\cdots l_M'}}.
\end{align}
Just as Leibnitz rule holds only for vectors in $\X(\A)$, as was shown
in Proposition \ref{prop:nablaProperties}, the corresponding rule for
tensors will hold as long as they are invariant under projections, i.e.
\begin{align}
  T^{k_1\cdots k_N}\D_{k_m}^l = T^{k_1\cdots k_{m-1} lk_{m+1}\cdots k_N},
\end{align}
for $m=1,\ldots,N$. Such tensors will be called \emph{tangential}, and
it is clear from the definition that the covariant derivative of a
tangential tensor is again a tangential tensor. For instance, one
computes that for tangential tensors $T_{kl}$ and $X^i$
\begin{align*}
  \nabla_i(T_{kl}X^l) &= \D_k^m\D_i\paraa{T_{ml}X^l}
  = \D_k^mT_{ml}\D_i(X^l) + \D_k^mX^l\D_i(T_{ml})\\
  &=T_{kl'}\D^{l'}_l\D_i(X^l) + \D^m_kX^{l'}\D_{l'}^l\D_i(T_{ml})\\
  &=T_{kl'}\nabla_iX^{l'} + X^{l'}\nabla_i(T_{kl'}).
\end{align*}

\noindent Note that the above definition in
(\ref{eq:covariantTensorDef}) coincides, for tangential tensors, with
the previous index-free definition. For instance, one easily computes
that
\begin{align*}
  \paraa{\nabla_XT}(Y,Z)\equiv\nabla_X\paraa{T(Y,Z)} 
  - T(\nabla_XY,Z)-T(Y,\nabla_XZ)
  = X^iY^jZ^k\nabla_iT_{jk}
\end{align*}
for $X,Y,Z\in\X(\A)$ and $T$ a tangential tensor. With this notation,
one computes that the following relation holds
\begin{align}\label{eq:curvatureInComponents}
  \paraa{R(X,Y)Z}^i = R^i_{jkl}Z^jX^kY^l =
  X^kY^l\paraa{\nabla_k\nabla_lZ^i-\nabla_l\nabla_kZ^i}.
\end{align}    

\noindent Let us also note that the Codazzi-Mainardi equations,
in the form of Proposition \ref{prop:CMequations} (with $M=\reals^m$)
\begin{align*}
  0 &=
  X_iY_jZ_k\bracketb{\D^i\D^k(N_A^j) - \D^j\D^k(N_A^i)
    +\D^k\D^j(N_A^i)-\D^k\D^i(N_A^j)
  },
\end{align*}
are satisfied due to Lemma \ref{lemma:DDSymmetric}.

It is a standard theorem in differential geometry that if the
sectional curvature only depends on the point (and not on the choice
of tangent plane) then the sectional curvature is constant (if the
dimension is greater than or equal to three). In the following, we
shall derive an analogous theorem for almost \KP\ algebras. Let us
first define the sectional curvature.

\begin{definition}
  Let $\A$ be an almost \KP\ algebra with curvature tensor $R$. For
  any $X,Y\in\X(A)$, the \emph{sectional curvature} is defined as
  \begin{align*}
    K(X,Y) = \frac{R(X,Y,X,Y)}{(X,X)(Y,Y)-(X,Y)^2}.
  \end{align*}
\end{definition}

\noindent For almost \KP\ algebras, if $K(X,Y)$ is independent of $X$
and $Y$, then the sectional curvature is in the center of the Poisson
algebra.

\begin{proposition}
  Let $\A$ be an almost \KP\ algebra with curvature tensor $R$ and
  geometric dimension $n\geq 3$. If $K(X,Y)=k\in\A$ for all $X,Y\in\X(\A)$
  then $\{k,u\}=0$ for all $u\in\A$.
\end{proposition}

\begin{proof}
  It is a standard algebraic result that if $R$ and $R'$ are two
  quadri-linear maps, satisfying (\ref{eq:firstBianchi}) and
  (\ref{eq:Rsymmetry1}), and $R(X,Y,X,Y)=R'(X,Y,X,Y)$ for all
  $X,Y\in\X(\A)$ then $R(X,Y,Z,V)=R'(X,Y,Z,V)$ for all
  $X,Y,Z,V\in\X(\A)$ \cite{kn:foundationsDiffGeometryI}. Hence, if we
  define
  \begin{align*}
    R'(X,Y,Z,V) = (X,Z)(Y,V)-(X,V)(Y,Z),
  \end{align*}
  for which $K(X,Y)=1$ for all $X,Y\in\X(\A)$, it follows that
  $R(X,Y,Z,V)=kR'(X,Y,Z,V)$. Since $\nabla$ is a metric connection one
  has $(\nabla_UR')(X,Y,Z,V)=0$ and
  \begin{align*}
    \paraa{\nabla_UR}(X,Y,Z,V) = \paraa{\nabla_U kR'}(X,Y,Z,V)
    =\nabla_U(k)R'(X,Y,Z,V).
  \end{align*}
  If we sum this identity over cyclic permutations of $U,X,Y$ the left
  hand side will vanish due to the second Bianchi identity and one is
  left with
  \begin{align*}
    0 &= \nabla_U(k)\parab{(X,Z)(Y,V)-(X,V)(Y,Z)}\\
    &\quad+\nabla_X(k)\parab{(Y,Z)(U,V)-(Y,V)(U,Z)}\\
    &\quad+\nabla_Y(k)\parab{(U,Z)(X,V)-(U,V)(X,Z)}.
  \end{align*}
  Given an arbitrary $X\in\X(\A)$ one can always find $Y,Z\in\X(\A)$
  such that $(X,Y)=(X,Z)=(Y,Z)=0$, since the geometric dimension
  of $\A$ is at least 3. For such vectors, the above relation becomes
  \begin{align*}
    -\nabla_X(k)(Y,V)(U,Z)+\nabla_Y(k)(U,Z)(X,V) = 0,
  \end{align*}
  and for $U=Z$ and $V=Y$ one obtains
  \begin{align*}
    \nabla_X(k)(Y,Y)(Z,Z) = 0,
  \end{align*}
  which implies that $\nabla_X(k)=0$. Thus, for all $X\in\X(\A)$ it
  holds that $X^i\D_i(k)=0$, which is equivalent to $\D^{li}\D_i(k)=0$
  for $l=1,\ldots,m$. Writing out this equation yields
  \begin{align*}
    0 = \D^{li}\D_i(k)=\D^l(k) = \frac{1}{\gamma^2}\P^{lm}{\P^i}_m(\d_ik),
  \end{align*}
  and multiplying by $\P_{lj}$ gives
  \begin{align*}
    0 = \frac{1}{\gamma^2}\P_{lj}\P^{lm}{\P^i}_m(\d_ik)
    ={\P^i}_j\d_ik = \{k,x_j\}
  \end{align*}
  which implies that $\{k,u\}=0$ for all $u\in\A$.
\end{proof}

\subsection{Tracial states, Stoke's theorem and orderings}

\noindent As an important illustration of how the developed algebraic
techniques can be used, we aim to prove that a bound on the Ricci
curvature induces a bound on the eigenvalues of the Laplace operator,
which is a standard result for compact Riemannian manifolds. To
achieve this goal we shall, in this section, introduce concepts of
integration and ordering on $\A$.

\begin{definition}
  Let $\A$ be an almost \KP\ algebra. A \emph{state}\footnote{A
    slightly more appropriate $\ast$-algebraic term is \emph{positive
      linear functional}, but for simplicity we have chosen a less
    cumbersome terminology.} on $\A$ is a $\complex$-linear map
  $\ints:\A\to\complex$ such that
  \begin{equation}
    \ints a^\ast = \overline{\ints a}
    \quad\text{and}\quad
    \ints a^\ast a \geq 0
  \end{equation}
  for all $a\in\A$.  
\end{definition}

\noindent In a noncommutative $\ast$-algebra, a state that fulfills
$\ints XY=\ints YX$, or $\ints[X,Y]=0$, is called a \emph{tracial
  state}. An analogous extension to commutative Poisson algebras would
be to require $\ints\{a,b\}=0$. However, what one needs for
calculations is an equation in correspondence with the fact that the
integral (over a closed manifold) of the divergence of a vector field
is zero, which allows one to perform ``partial integration''.  In
Section \ref{sec:integration} it was shown that
$\int_\Sigma\nablah_iX^i=0$, which motivates the following definition.

\begin{definition}
  Let $\ints$ be a state on an almost \KP\ algebra. The state is
  called \emph{tracial} if
  \begin{align}
    \ints \nabla_iX^i = 0\label{eq:traceTotalDerivProp}
  \end{align}
  for all $X\in\X(\A)$.
\end{definition}

\noindent Note that in the case when a square root of $\gamma^2$
exists, and the algebra fulfills the additional condition
(\ref{eq:algebraCSparallell}), a tracial state in a \KP\ algebra
fulfills
\begin{align*}
  \ints \frac{1}{\gamma}\{x^i,X_i\} = 0,
\end{align*}
which is in analogy with a tracial state in a noncommutative algebra.

\begin{definition}
  An almost \KP\ algebra with a tracial state is called a
  \emph{geometric almost \KP\ algebra}.
\end{definition}

\noindent As usual, any state induces a $\complex$-valued sesquilinear form on $\A$ via
\begin{equation}
  \angles{u,v} = \ints u^\ast v,
\end{equation}
which may be extended to $\TSigma$ by setting 
\begin{equation*}
  \angles{X,Y} = \ints (X^\ast,Y).
\end{equation*}

\noindent Let us now introduce a preorder on $\A$.

\begin{definition}
  Let $\A$ be an almost \KP\ algebra. We say that $a$ is
  \emph{positive}, and write $a\geq 0$ if and only if
  \begin{align}\label{eq:partialOrderDef}
    a = \frac{\sum_{i=1}^N u_i^\ast u_i}
    {\sum_{k=1}^{N'} v_k^\ast v_k},
  \end{align}
  for some elements $u_i,v_k\in\A$. Moreover, we write $a\geq b$ whenever
  $a-b\geq 0$.
\end{definition}

\noindent Let us state some of the properties of this ordering.

\begin{proposition}
  The relation $\geq$ has the following properties
  \begin{enumerate}
  \item $a\geq a$,
  \item if $a\geq b$ and $b\geq c$ then $a\geq c$,
  \item if $a\geq b$ then $a+c\geq b+c$ for all $c\in\A$,
  \item if $a\geq 0$ and $b\geq 0$ then $ab\geq 0$,
  \end{enumerate}
  i.e. $\geq$ is a ring preorder.
\end{proposition}

\begin{proof}
  It is immediate from (\ref{eq:partialOrderDef}) that sums and
  products of positive elements are again positive elements.
  \textit{Reflexivity: $a\leq a$}. This is equivalent to $a-a\geq 0$,
  which is true since $0=0^\ast 0$. \textit{Transitivity.} Assume that
  $a\geq b$ and $b\geq c$. By definition, this means that $a-b$ and
  $b-c$ are positive, which implies that their sum, i.e. $a-c$, is
  positive. Thus, $a\geq c$. 
\end{proof}

\noindent In particular, one notes that $(X^\ast,X)\geq 0$ for all
$X\in\Der(\A)$, which enables us to prove the Cauchy--Schwarz
inequality for tensors.

\begin{proposition}
  Let $T^{ij}$ be a tangential tensor. Then it holds that
  \begin{align*}
    \paraa{\Tr\D}^2(T^{ij})^\ast T_{ij}\geq
    \paraa{\Tr\D}(\Tr T)^\ast(\Tr T).
  \end{align*}
\end{proposition}


\begin{proof}
  First, we note that for a tangential tensor it holds that
  \begin{align*}
    \Tr T = T^i_i = \D^{ik}T_{ik},
  \end{align*}
  and writing
  \begin{align*}
    (\Tr\D)T^{ij} = (\Tr\D)T^{ij}-(\Tr T)\D^{ij}+(\Tr T)\D^{ij},
  \end{align*}
  it follows that (recall that $(\D^{ij})^\ast=\D^{ij}$)
  \begin{equation}\label{eq:Tsquareproof}
    \begin{split}
      (\Tr\D)^2(T^{ij})^\ast T_{ij} &= \parab{(\Tr\D)T^{ij}-(\Tr T)\D^{ij}}^\ast
      \parab{(\Tr\D)T_{ij}-(\Tr T)\D_{ij}}\\
      &\qquad\qquad
      +(\Tr T)^\ast(\Tr T)(\Tr\D)
    \end{split}
  \end{equation}
  since
  \begin{align*}
    \parab{(\Tr\D)T^{ij}-(\Tr T)\D^{ij}}^\ast(\Tr T)\D_{ij} = 0.
  \end{align*}
  Using that the first term in (\ref{eq:Tsquareproof}) is positive
  gives
  \begin{align*}
    \paraa{\Tr\D}^2(T^{ij})^\ast T_{ij}\geq
    \paraa{\Tr\D}(\Tr T)^\ast(\Tr T).
  \end{align*}
  which completes the proof.
\end{proof}

\noindent If the geometric dimension of the algebra is $n$, the
inequality can be written as
\begin{align*}
  (T^{ij})^\ast T_{ij} \geq \frac{1}{n}(\Tr T)^\ast(\Tr T).
\end{align*}
In the following the above inequality will be applied to the tensor
$\nabla_i\nabla_j(u)$, where $u$ is a hermitian element of $\A$, which
gives
\begin{align}\label{eq:algebraicTraceIneq}
  \nabla^i\nabla^j(u)\nabla_j\nabla_i(u)\geq
  \frac{1}{n}\paraa{\nabla^i\nabla_i(u)}^2.
\end{align}

\subsection{Eigenvalues of the Laplace operator}

\noindent Let us now proceed to define the Laplace operator, and to
show that its eigenvalues are bounded by the Ricci curvature. We start
by introducing the Laplace operator, together with some of its properties.

\begin{definition}
  The operator $\Delta:\A\to\A$, defined as
  \begin{equation}
    \Delta(u) = \nabla^i\nabla_i(u),
  \end{equation}
  is called the \emph{Laplace operator on $\A$}. An \emph{eigenvector
    of $\Delta$} is an element $u\in\A$ such that $\Delta(u)=\lambda
  u$ for some $\lambda\in\complex$. The complex number $\lambda$ is
  then called an \emph{eigenvalue} of $\Delta$.
\end{definition}

\begin{proposition}
  In a geometric almost \KP\ algebra, the Laplace operator is a
  self-adjoint operator with respect to the sesquilinear form
  $\angles{\cdot,\cdot}$. Hence, for any eigenvector $u$ with
  $\angles{u,u}>0$, the corresponding eigenvalue is real.
\end{proposition}

\begin{proof}
  Since $\nabla^i(u)^\ast = \nabla^i(u^\ast)$ it follows that
  $\Delta(u)^\ast = \Delta(u^\ast)$. As the state is tracial, it
  follows that
  \begin{align*}
    \angles{\Delta(u),v} = \ints \nabla^i\nabla_i(u^\ast) v 
    = -\ints\nabla^i(u^\ast)\nabla_i(v) = \ints u^\ast\Delta(v)
    = \angles{u,\Delta(v)}.
  \end{align*}
  Let $u$ be an eigenvector of $\Delta$ with eigenvalue
  $\lambda$. Then it holds that
  \begin{align*}
    \bar{\lambda}\angles{u,u} = \angles{\Delta(u),u} = 
    \angles{u,\Delta(u)} = \lambda\angles{u,u},
  \end{align*}
  from which it follows that $\lambda=\bar{\lambda}$ since
  $\angles{u,u}>0$.
\end{proof}

\noindent Without any further assumptions on the algebra, eigenvectors
of the Laplace operator may in general have degenerate features. Let
us therefore restrict to a particular class of eigenvectors.

\begin{definition}
  Let $u$ be an eigenvector of the Laplace operator with eigenvalue
  $\lambda$. The eigenvector is called \emph{non-degenerate} if
  $\angles{u,u}>0$ and $\angles{\nabla(u),\nabla(u)}>0$.
\end{definition}

\begin{proposition}
  Let $u$ be a non-degenerate eigenvector of the Laplace operator with
  eigenvalue $-\lambda$. Then it follows that $\lambda>0$.
\end{proposition}

\begin{proof}
  One computes that
  \begin{align*}
    \lambda\angles{u,u} = -\angles{u,\Delta(u)}
    =-\ints u^\ast\Delta(u) = \ints\nabla_i(u)^\ast\nabla^i(u)
    =\angles{\nabla(u),\nabla(u)}.
  \end{align*}
  Since $u$ is assumed to be non-degenerate, it holds that both
  $\angles{u,u}$ and $\angles{\nabla(u),\nabla(u)}$ are strictly
  positive, which implies that $\lambda>0$.
\end{proof}

\noindent We shall now prove, in the purely algebraic setting, a
classical theorem of differential geometry saying that a bound on the
Ricci curvature induces a bound on the eigenvalues of the Laplace
operator (corresponding to non-degenerate eigenvectors) on a compact
manifold.

\begin{theorem}
  Let $\A$ be a geometric almost \KP\ algebra with geometric dimension
  $n\geq 2$, and let $-\lambda\neq 0$ be an eigenvalue of the Laplace
  operator corresponding to a non-degenerate eigenvector $u$. If there
  exists a real number $\kappa>0$ such that $\R(X^\ast,X)\geq
  \kappa(X^\ast,X)$ for all $X\in\X(\A)$ then $\lambda\geq
  n\kappa/(n-1)$.
\end{theorem}

\begin{proof}
  First we note that one can always choose the eigenvector $u$ to be
  hermitian, since $u^\ast$ is also an eigenvector of $\Delta$ with
  eigenvalue $\lambda$.  Let us start by writing
  \begin{align}\label{eq:compDelta}
    \ints \paraa{\Delta(u)}^2 =-\lambda\ints u\Delta(u)
    =-\lambda\ints u\nabla^i\nabla_i(u) 
    = \lambda\ints\nabla^i(u)\nabla_i(u),
  \end{align}
  since the state is assumed to be tracial.  One the other hand one
  gets
  \begin{align*}
    \ints\paraa{\Delta(u)}^2 = \ints\nabla^i\nabla_i(u)\nabla^k\nabla_k(u)
    =-\ints\nabla_k(u)\nabla^k\nabla^i\nabla_i(u),
  \end{align*}
  and using equation (\ref{eq:curvatureInComponents}) gives
  \begin{align*}
    \ints\paraa{\Delta(u)}^2 = -\ints\bracketb{
      \nabla_k(u)\nabla^i\nabla^k\nabla_i(u)
      -\R(\nabla u,\nabla u)}.
  \end{align*}
  After partial integration one obtains
  \begin{align*}
    \ints\paraa{\Delta(u)}^2 = \ints\bracketb{
    \nabla_i\nabla_k(u)\nabla^k\nabla^i(u)+\R(\nabla u,\nabla u)}.
  \end{align*}
  Now, using the inequality (\ref{eq:algebraicTraceIneq}) together
  with the assumption that $\R(X^\ast,X)\geq\kappa(X^\ast,X)$ for all
  $X\in\X(\A)$ gives
  \begin{align*}
    \ints\paraa{\Delta(u)}^2&\geq 
    \frac{1}{n}\ints\paraa{\Delta(u)}^2
    +\kappa\ints\nabla^i(u)\nabla_i(u)\\
    &=\parab{\frac{\lambda}{n}+\kappa}\ints\nabla^i(u)\nabla_i(u),
  \end{align*}
  Now, we compare this expression with (\ref{eq:compDelta}) and
  conclude that
  \begin{align*}
    \frac{1}{n}\paraa{\lambda(n-1)-n\kappa}\ints\nabla^i(u)\nabla_i(u)
    =\frac{1}{n}\paraa{\lambda(n-1)-n\kappa}\ints\nabla^i(u)\nabla_i(u)\geq 0,
  \end{align*}
  which implies $\lambda\geq n\kappa/(n-1)$ since $u$ is a
  non-degenerate eigenvector.
\end{proof}

\section{Examples}\label{sec:examples}

\noindent Let us consider two examples of \KP\ algebras that are
constructed in an algebraic way, although they have clear geometrical
interpretations.

\subsection{A simple flat example}

\noindent Let $\A$ be generated by
\begin{align*}
  \{x^i\}=\{p^1,\ldots,p^n,q^1,\ldots,q^n,n^1,\ldots,n^p\}
\end{align*}
and we shall let indices $a,b,c,\ldots$ run from $1$ to $n$ and
indices $A,B,C,\ldots$ from $1$ to $p$. We introduce a Poisson
structure defined by
\begin{align*}
  &\{p^a,p^b\}=\{q^a,q^b\}=\{n^A,x^i\}=0\\
  &\{p^a,q^b\}=\delta^{ab}\gamma\cdot 1 
\end{align*}
with $\gamma\in\reals$. It is easy to check that it is a \KP\ algebra
with characteristic function $\gamma^2\cdot 1$. The projection
operators $\D,\Pi$ become
\begin{align*}
  (\D_{ik})=\diag(\underbrace{1,\ldots,1}_{2n},0,\ldots,0)\qquad
  (\Pi_{ik})=\diag(0,\ldots,0,\underbrace{1,\ldots,1}_p),
\end{align*}
from which it follows that $\X(\A)$ is $2n$-dimensional and a basis is
given by $\d_1,\ldots,\d_{2n}$. Hence, $\D_i=\d_i$ and $R(X,Y,Z,V)=0$
for all $X,Y,Z,V\in\X(\A)$.

\subsection{Algebras defined by a polynomial}

\noindent Let us introduce a Poisson algebra which has been used to
construct matrix regularizations of surfaces
\cite{abhhs:fuzzy,abhhs:noncommutative,a:repcalg,as:affine,a:phdthesis}.
Let $\A=\complex[x^1,x^2,x^3]$ be the polynomial algebra in three
variables together with the Poisson structure
\begin{align*}
  \{x^i,x^j\} = \eps^{ijk}\d_kC
\end{align*}
where $C$ is an arbitrary (hermitian) element of $\A$, and
$\eps^{ijk}$ is the totally anti-symmetric Levi-Civita symbol. It is
easy to check that $\A$ is an almost \KP\ algebra with
\begin{align*}
  \gamma^2 = (\d_1C)^2+(\d_2C)^2+(\d_3C)^2.
\end{align*}
Moreover, one also can check that $\A$ is a \KP\ algebra. The
projection operator $\D_{ik}$ is computed to be
\begin{align*}
  \D_{ik} = \delta_{ik}-\frac{1}{\gamma^2}\paraa{\d_iC}\paraa{\d_kC},
\end{align*}
which gives $\Pi_{ik} = (\d_iC)(\d_kC)/\gamma^2$. Hence, the geometric
dimension of $\A$ is $2$, and a basis for $\N(\A)$ (which is then
one-dimensional) is given by $\sum_{i=1}^3(\d_iC)\d_i$. By using Gauss
formula (in Proposition \ref{prop:GaussFormula}), one computes the
curvature to be
\begin{align*}
  R(X,Y,Z,V) = \frac{1}{\gamma^2}\parab{\paraa{\d^2_{ik}C}\paraa{\d^2_{jl}C}
  -\paraa{\d^2_{il}C}\paraa{\d^2_{jk}C}}X^iY^jZ^kV^l.
\end{align*}
For instance, choosing $C=x^2+y^2+z^2-r^21$, with $r\in\reals$, one
computes that $\A$ has constant curvature, i.e. $K(X,Y)=1/\gamma^2$
for all $X,Y\in\X(\A)$. Note that if one considers the quotient
algebra $\A\slash\angles{C}$ (to which the Poisson structure
restricts), the sectional curvature will be a constant,
i.e. $K(X,Y)=1/r^2$.

\section*{Acknowledgment}

\noindent J. A. would like to thank J. Hoppe and H. Steinacker for discussions.

\bibliographystyle{alpha}
\bibliography{kahlerpb}

\end{document}